\documentclass[12pt]{amsart}
\makeatletter
\def\ps@plain{%
\def\@oddfoot{\hfil\thepage\hfil}%
\def\@evenfoot{\hfil\thepage\hfil}%
\def\@oddhead{}%
\def\@evenhead{}%
}

\makeatother
\usepackage[utf8]{inputenc}
\usepackage{amssymb}
\usepackage{amsmath}
\usepackage{bbm}
\usepackage{float}
\usepackage{graphicx}
\usepackage{subcaption}
\usepackage{tikz-cd}
\usepackage{amsthm}
\usepackage[colorlinks=true, allcolors=blue]{hyperref}
\usepackage[letterpaper,top=2cm,bottom=2cm,left=3cm,right=3cm,marginparwidth=1.75cm]{geometry}
\usepackage{enumitem}   
\usepackage{cite}
\usepackage{booktabs}

\newtheorem{theorem}{Theorem}[section]

\newtheorem{definition}[theorem]{Definition}
\newtheorem{lemma}[theorem]{Lemma}

\newtheorem{proposition}[theorem]{Proposition}
\newtheorem*{prop}{Proposition}

\newtheorem{remark}[theorem]{Remark}

\DeclareMathOperator{\sign}{sign}

\numberwithin{theorem}{section}

\newcommand{\Pb}{\mathbb{P}}

\newcommand{\R}{\mathbb{R}}

\newcommand{\bs}{\bigskip}
\newcommand{\ms}{\medskip}
\newcommand{\dd}{\mathrm{d}}

\newcommand{\1}{\mathbbm{1}}

\newcommand{\F}{\mathcal{F}}

\newcommand{\Jac}{\operatorname{Jacobi}}

\title{A one-parameter family of local-time conditioned processes : the generalized Brownian burglars}
\author{Arthur Dremaux}
\address{DPMMS, University of Cambridge}
\email {ad2231@cam.ac.uk}
\date{}

\begin{document}

\begin{abstract}
We construct and characterize the one-parameter family of real-valued processes that allow to recover the law of a Brownian loop-soup of any intensity on the real line conditionally on its occupation time field. These processes generalize the Brownian burglar constructed by Warren and Yor (that corresponds to our process in the limiting case where the intensity vanishes and there is just one Brownian motion). These processes are closely related to the Bass-Burdzy flow and recent work of A\"id\'ekon, Hu and Shi on the stochastic Jacobi flow.

Our approach uses a formalism, building on the notion of driver processes, where one considers the evolving object to be the line that we ``stretch", instead of the burglar itself and its remaining occupation time.
This leads to a characterization of these (generalized) Brownian burglars by simple natural axioms. It also allows to treat ``negative intensities'' and provides fairly direct derivations of several properties of the burglars, and in particular a ``target-independence/locality" property at the special intensity where the local time is the square of a Gaussian Free Field.
\end{abstract}

\maketitle
\section{Introduction}

Let $B$ be a real-valued Brownian motion defined on a time-interval $[0,T]$. One of its important properties is the existence and bicontinuity of Paul L\'evy's local time process $L_B(t,x)$ verifying the occupation-time formula for any bounded Borel function $f$ (see for instance the textbooks \cite{levy1992processus,revuzyor,mp2010}) $$\int_0^t f(B_s)\dd s=\int_\R f(r)L_B(t,r)\dd r.$$

The Brownian loop-soup $\mathcal L$ of intensity $\nu \ge 0$ on some interval of the line is a Poisson point process of Brownian loops, with intensity $\nu$ times the natural Brownian loop measure.
In any fixed bounded interval, it contains almost surely infinitely many loops, but only finitely many with a diameter greater than any fixed $\varepsilon$.
These loop-soups have been first defined and studied in the continuum space by Lawler and Werner \cite {MR2045953} and then on transient graphs and cable-graphs by Le Jan and Lupu \cite {LeJan2011,Titus_Lupu_2018}. In our case, where the cable-graph is a proper subset of $\R$, the loop-soup can also be implicitly detected in the background of some of the sample path decompositions of Brownian paths in works by Pitman and/or Yor.
One can also define the occupation time profile of each of these loops, and by summing all these contributions, one obtains the occupation time profile $\Lambda$ of the whole collection of Brownian loops, which turns out to be also almost surely finite and continuous.
Then, on e.g.\ the interval $(0, \infty)$ (see \cite{Titus_Lupu_2018}), one has the loop-soup analog to the classical Ray-Knight Theorems: the local-time profile $\Lambda$ is a squared Bessel process of dimension $2\nu$ -- the additivity property of the squared Bessel processes corresponds to the additivity property of loop-soups (the union of two independent loop-soups is another loop-soup of intensity $\nu + \nu'$), and the local time of each individual loop can be viewed via the decomposition of the squared Bessel process as a Poisson point process of squared Bessel excursions of dimension $0$.
At the critical intensity $\nu=1/2$, it has the law of a squared Brownian motion on the half-line (which can be viewed as a squared GFF) as stated by ``Le Jan's isomorphism" \cite{LeJan2011}.

So, from a Brownian loop-soup with intensity $\nu$, one can construct its total occupation time $\Lambda$, which turns out to be a squared Bessel process of dimension $2\nu$. This leads naturally to the reverse question of how to describe the law of the loop-soup, when it is conditioned by the total local time profile $\Lambda$, which is the main theme of the present paper.

A very closely related question (as we shall explain, it corresponds to the limiting case where $\nu \to 0+$) was first studied and answered by Warren and Yor in \cite{SPS_1998__32__328_0}, where they defined the process (that they called the Brownian burglar process) with the law of one Brownian motion (up to some stopping time say) conditioned by its occupation time $\Lambda$. Note that this process has the same almost sure pathwise properties as usual Brownian motion (e.g.\ quadratic variation etc), but (see our previous work \cite{Dremaux2026}) it turns out that its law on any time-interval is actually singular with respect to the law of the usual Brownian motion.
Aspects of this general question have been treated in a number of papers including \cite {AHS,Lupu_2019,Lupu_2021}, sometimes with somewhat different motivations, goals and techniques -- and we will describe some of them in the next few paragraphs.

\ms

One main contribution of the present paper is arguably the formalism itself that we use to tackle this question (even if it of course relates to earlier work as we will explain) that we now briefly discuss:

As a warm-up, let us first make the following simple observation:  if $g$ is an increasing diffeomorphism between two open sub-intervals $I$ and $J$ of $\R$, and if a Brownian motion $B$ takes values in $I$ during the interval $[0,T]$, then the corresponding process $X = g(B)$ has increasing quadratic variation, and a bicontinuous local time process $L_X$ verifying the occupation-time formula $$\int_0^t f(X_s)\dd[X]_s=\int_J f(r)L_X(t,r)\dd r,$$
and this local-time process $L_X$ is simply described by the ``stretching" formula
$$ L_X(t,g(r)) = g'(r)L_B(t,r).$$
Note that here the convention to ``define'' the local time of a process is not that of its occupation-time density, but rather indexed by its quadratic variation, which can be considered to be intrinsic to the geometric path of the process, so that the local-time profile of a process $X$ at final time $T$ will actually be independent of the chosen time-parametrization. We will stick with this important convention until the end of the paper, so that all the processes we will be dealing with really have to be understood as geometrical paths, and \textbf{will always implicitly be defined up to a time-reparametrization}.

With this point of view, the diffeomorphism $g$ is a geometric transformation that stretches both the space-scale and the local-time by a local factor $g'(r)$, so that it preserves the local resistance metric $\dd r / L(r)$. For the rest of the article, we will usually keep this geometric point of view and call such transformations \textbf{stretching morphisms}. 
They act in the same way on the whole loop-soup, and stretch it to a collection of loops with total local-time profile $\tilde \Lambda(g(r)) = g'(r)\Lambda(r)$. Notice that the obtained collection of loops can be very singular to a usual loop-soup, since e.g.\ there exists a (unique up to translation) stretching $g$ mapping the loops of $\mathcal L$ in the cluster of $x\in \R$ to a collection of loops with local time profile constant equal to 1 on the whole of $\R$. We will call all such local time profiles attainable by such stretchings \textbf{admissible capacities}.

\ms

The starting point to our work is the following important remark, first noticed by Warren and Yor in \cite{SPS_1998__32__328_0} for the case of the usual Brownian burglar, and also adapted to general positive $\nu$ by A\"id\'ekon, Hu and Shi in \cite{AHS}. The law of the Brownian burglar processes exhibits a \textbf{stretching covariance}:

\begin{prop}
Let $\mathcal L$ be a Brownian loop-soup of intensity $\nu\geq 0$ in $\R\setminus\{x_0\}$, and let $B$ be a usual Brownian motion started from $x$ and ran up to the stopping time $\tau$ at which it reaches local time $a>0$ at $x_0$. Let $\Lambda$ be total local-time profile of the union of $\mathcal L$ and $B$. Let $g$ be the unique stretching morphism sending the point $x$ to 0 and mapping (the capacity) $\Lambda$ to the constant capacity $\rho=1$ on $\R$, i.e.\ $g : z \mapsto \int_x^z \frac{\dd r}{\Lambda(r)}$.
Then, the law of the stretched process $g(B)$ (defined up to time-reparametrization) is independent of $\Lambda$ (and of $a$) conditionally on $g(x_0)$.
\end{prop}

We will call the Brownian motion $B$ conditioned by the total occupation time profile $\Lambda$ the \textbf{(generalized) Brownian burglar of intensity $\nu$}, initial capacity $\Lambda$, started from $x$ and targeting $x_0$.
This proposition can immediately be extended to any other fixed admissible capacity $\rho$ than the constant capacity 1, and thus gives a natural way of properly defining the law of this general burglar for any admissible local time profile $\rho$, and the stretching covariance then translates in the following way :

\begin{prop}[Stretching covariance of the burglars]
Let $\Lambda$ be an admissible capacity on an interval $I$, and let $x,x_0 \in \overline{I}$. Let $X$ be a Brownian burglar of intensity $\nu \in\R$ and initial capacity $\Lambda$, starting from $x$ and targeting $x_0$. Let $g$ be a stretching morphism mapping $\Lambda$ to an other admissible capacity $\rho$ on some interval $J$.
Then, the process $g(X)$ has (up to time-reparametrization) the law of a Brownian burglar of intensity $\nu$ and initial capacity $\rho$, starting from $g(x)$ and targeting $g(x_0)$.
\end{prop}

It may be useful here to draw a parallel with Schramm--Loewner Evolutions. The one-parameter family of SLE$_\kappa$ curves is uniquely characterized by their conformal covariance (up to time-reparametrization), in addition to natural properties for such curves (domain Markov property, and symmetry). We will indeed formulate a counterpart for the Brownian burglars, in which stretching covariance plays the same role as conformal invariance for SLEs.
Namely, our Theorem \ref{thm:classification} will state that {\em this stretching covariance property, in addition to some natural symmetry properties and an analogue of the domain Markov property, characterize the family of laws of the burglars up to the real parameter $\nu$}. In particular, this also defines a family of space-filling -- or rather ``local-time filling" -- burglar processes for every negative intensities $\nu \leq 0$.

This analogy goes somewhat further and suggests using a similar construction of the burglar processes to the SLE one. We will indeed give a formal construction of the burglars for any real parameter $\nu\in\R$, any initial capacity profile and up to full-time via a driver process and a continuous flow of stretching maps.

In short, the idea is the following : the law of a burglar $X$ after time $t$ is the law of an independent burglar in the remaining capacity $\lambda_t = \lambda - L_X(t,\cdot)$, started from $X_t$ -- this is the \textbf{local-time Markov property}. By the stretching covariance, if we stretch via a map $g_t$ the remaining capacity $\lambda_t$ to the constant capacity 1, the law of $\left(g_0^{-1}(g_t(X_{t+s}))\right)_{s\ge 0}$ is that of a burglar in the original capacity $\lambda$, started from $g_0^{-1}(g_t(X_t))$ and independent from the burglar in $[0,t]$. Thus, in some sense, if we continually stretch the capacity $\lambda_t$ back to the constant capacity 1, the corresponding process -- called the \textbf{driver process} -- $(D_t)_{t\geq 0}:=(g_t(X_t))_{t\geq 0}$ will be a Markovian diffusion. Finally, we actually show that such a driver process uniquely determines the initial process $X$ on the capacity $\lambda$, and thus we define the burglars via their drivers.

\ms

Heuristically, if we formally apply Ito's formula to the burglars $X$ via their driver $D$ (which are semimartingales) and the $\mathcal C^1$ stretching maps $g_t$, they are (up to time reparametrization) solution to the singular SDE :
\begin{equation}
\label{heursde}
\dd X_t = \dd B_t + \frac{1}{2}\frac{``\partial_x \lambda_t(X_t)"}{\lambda_t(X_t)}\dd t + \frac{1/2 - \nu}{\lambda_t(X_t)} \text{sign}(X_t - x_0) \dd t.
\end{equation}
Even though $\lambda_t(\cdot)$ is usually not differentiable, stretching back this SDE to the constant capacity $1$ precisely via the stretchings $g_t$ allows us to understand this term. It is usual to make sense of $\partial_x\lambda$ when $\lambda$ is non-differentiable (but constant in $t$) via a scale-change (see \cite{revuzyor}). Here the same has to be understood for the term $\partial_x\lambda_t$, but via the time-dependent change of scale $g_t$.
The same kind of SDE already appears in \cite{Lupu_2019,Lupu_2021,Alberts_2013}, and it should be the scaling limit of the jump processes defined in \cite{10.1214/24-EJP1176} to invert Ray-Knight Theorems on trees.

Analyzing this heuristic SDE give us a good first idea of how the burglars should behave. First, when the burglar $X$ is at some point $x$, it only ``sees" two things :
the remaining capacity / local time profile $\lambda_t$ around $x$ (in a neighborhood as small as we want), and whether its target $x_0$ is on the left or on the right side of him.
Indeed, $X$ is always pushed towards places with more local time available (via the singular term), and depending on the intensity $\nu$ it also gets an additional drift term attracting it towards (if $\nu>1/2$) or pushing it away (if $\nu<1/2$) from the target $x_0$. We call this property \textbf{side-locality}. At the critical intensity $\nu=1/2$, this drift term disappears, suggesting a particularly nice ``target-independence" behaviour (analogous to that of SLE$_6$) : burglars with different targets $x$ and $y$ can be coupled until the disconnection time of $x$ and $y$ in $\lambda_t$. This result is partly present in the article \cite{Lupu_2021}, in which it is shown that this diffusion is the scaling limit of the so-called "Reversed Vertex Reinforced Jump Process", which is already target-independent in the discrete setup. This echoes some work at the critical intensity, in particular \cite{aidekon2020clusterexplorationsloopsoup} and the recent ``switching property" of the critical Brownian loop-soup derived in \cite{werner2025switchingidentitycablegraphloop}. In fact (see \cite {ADW}), one can use the present construction in order to derive the explicit switching-bijection at the level of loop-soups of critical intensity on general cable-graphs via ``peeling'', as outlined in Section 4.4 of \cite {werner2025switchingidentitycablegraphloop}.

Finally, some computations and the similarity with the process studied in \cite{Rogers01121994} suggest that the drift of the solution $X$ (our burglar) actually has non-trivial 4/3-variation of the form $\gamma \int_0^t \frac{\dd s}{\lambda_s(X_s)^{2/3}}$ for some constant $\gamma$ (which in particular implies that the burglars are not semi-martingales). We plan to address this question in
\cite {ADb}.

\ms 

One benefit of our framework is that it is direct, well-defined and  quite tractable. It provides a unified (and full-time) construction of the processes discussed in \cite{SPS_1998__32__328_0,MR1650567,Lupu_2021,AHS}, and the aforementioned analogy with SLEs leads to a number of burglar-analogs of many interesting objects/properties/features constructed/studied in the SLE framework (such as $\text{SLE}(\kappa,\rho)$ processes, loop ensembles, exploration trees, or space-filling SLEs of intensity $4<\kappa<8$).

\ms

The organization of this paper is the following. We start by introducing all the necessary formalism of the capacities, stretching maps, flows and driving processes in Section \ref{formalism}. In Section \ref{properties}, we then prove that the burglars defined through this formalism verify symmetry, stretching covariance, side-locality, and local-time Markov property, which characterize their law up to a real parameter $\nu\in\R$ by Theorem \ref{thm:classification}. We also give a very natural interpretation of the link of burglars with Bass-Burdzy flows (discussed in \cite{AHS, Lupu_2019, Lupu_2021}) via our drivers and stretching flows approach. We derive the target-independence property at the critical intensity $\nu=1/2$.
Finally, in Section \ref{disintegration}, we show how this formalism relates to the burglars already studied in previous works, and allows to disintegrate general loop-soups conditionally on their local time profiles.

\section{Formalism and construction of the burglars}
\label{formalism}

To study processes consuming some available local time on an interval, one must first equip these intervals with a corresponding underlying geometric structure (that is in some sense analogous to the conformal analysis setup in the complex plane).

\subsection{Capacities and stretching morphisms} 
For an open (not necessarily bounded) interval $I=(a,b)\subset \R$, we call \textit{local conductance} or \textit{capacity} any positive continuous function $\lambda:I\to \R_+^*$, and we will say that the capacity $\lambda$ is \textit{admissible} (on the interval $I$) if for some (and hence for all) $x_0\in I$, the total resistances on each side of $x_0$ are infinite, i.e. $$R_\lambda(\inf I,x_0)=\int_{\inf I}^{x_0}\frac{1}{\lambda} = +\infty \qquad \text{and} \qquad R_\lambda(x_0,\sup I) = \int_{x_0}^{\sup I}\frac{1}{\lambda} = +\infty.$$

We will consider space transformations that preserve the local resistance metric $R(x,x+ \dd x) = \dd x/\lambda(x)$. For two intervals $I,J$ and two corresponding admissible capacities $\lambda,\rho$ on $I$ and $J$, we will call \textit{Stretching morphism} from $\lambda$ to $\rho$, any increasing diffeomorphism $g$ from $I$ to $J$ such that for all $x,y\in I$, $$R_{\lambda}(x,y)=R_{\rho}(g(x),g(y)),$$ i.e. in an integral form: $$\int_x^y \frac{\dd u}{\lambda(u)}=\int_{g(x)}^{g(y)} \frac{\dd u}{\rho(u)}.$$ This is equivalent to locally verifying the differential relation $$g'(x)=\rho(g(x))/\lambda(x), \qquad \text{for all }x\in I.$$

This implies that if we fix two anchors $x_0\in I$ and $z_0\in J$, then there exists a unique stretching $g$ from $\lambda$ to $\rho$ such that $g(x_0)=z_0$. Stretching morphisms between admissible capacities with given anchors therefore have a groupoid structure. 
Such diffeomorphisms verify $$\lim_{x\to \inf I} g(x) = \inf J \quad \text{and} \quad \lim_{x\to \sup I} g(x) = \sup J,$$
and will therefore always be implicitly extended as functions between the compactified intervals $\overline I := I\cup\{\inf I, \sup I\}$ and $\overline J:=J\cup\{\inf J, \sup J\}$. When $g$ is a stretching morphism from capacity $\lambda$ to $\rho$, we will sometimes write that $\rho = g(\lambda)$.

\ms

A first central example of capacity is the constant (or uniform) capacity $\rho = 1$ on $\R$. It is admissible, and the stretching morphisms from any admissible capacity $\lambda$ on an interval $I$ to the uniform capacity are of the form:
\begin{equation}
\label{stretchunif}
\eta_{\lambda,x_0}(x)=\int_{x_0}^x \frac{\dd u}{\lambda(u)}, \qquad \text{for }x\in I.
\end{equation}
They are unique up to the choice of an anchor $x_0$. The corresponding Möbius transforms (i.e. the stretching morphisms from this capacity to itself) are simply the translations.

Another family of admissible capacities are the squared Bessel bridges. By usual properties of squared Bessel processes (see \cite{Lupu_2021,AHS}), we can define excursions of a squared Bessel process of any dimension $2\nu\in \R$ on some interval $I$, and the corresponding capacity $\lambda$ is almost surely admissible on $I$. Such an excursion corresponds for instance to the total local-time profile of some cluster of a loop-soup of intensity $\nu>0$.

\subsection{Admissible processes}
For any given capacity $\lambda_0$ on an interval $I$, we will say that a process $X$ defined on $\overline I$ up to a random final time $T(X)$ is \textit{admissible} for $\lambda_0$ if $X$ has strictly increasing quadratic variations and admits a bicontinuous local time profile $L_X$, which is dominated by the capacity $\lambda_0$, i.e. for all $x\in I$, $$L_X(T(X),x)\leq \lambda_0(x).$$ Here and in all this article, we only consider the process $X$ as a ``geometric path", and thus always understand $X$ up to some time-reparametrization (that would describe the same geometric path, but walked at a different speed). The local time $L_X$ of a process $X$ is then always defined via the occupation density-formula $$\int_0^t f(X_s)\dd[X]_s=\int_I f(r)L_{X}(t,r)\dd r,$$
which does not depend on the time-parametrization of $X$. In particular if $X$ is admissible for $\lambda_0$, any time-change of $X$ is also admissible so that this notion is well-defined.
When we will make the time-parameter $t$ appear in the rest of the article, it will always be relatively to the chosen time-parametrization for $X$. 

In particular, we will also sometimes discuss ``final times'' or ``stopping times'' of such paths / processes. These must also always be understood up to time reparametrization, and are indeed well-defined since e.g.\ up to time-change, a stopping time for the filtration of a process is also a stopping time in the time-changed filtration.

\ms

For such an admissible process (and any corresponding time-parametrization), we call the capacity $\lambda_t(\cdot):=\lambda_0(\cdot)-L_X(t,\cdot)\geq 0$ the ``capacity left at time $t$'', such that the process after time $t$, $(X_{t+\cdot})$ will be admissible for the capacity $\lambda_t$. 

\ms

Finally, for any target point $x_0\in \overline I$, we define for every $x\in I$ the disconnection time of $x$ from $x_0$ by $X$, $$T(x,x_0)=\inf\{t\geq0,|R_{\lambda_t}(x,y)|=+\infty\text{ for some } y\in (x,x_0)\},$$
and the non-increasing intervals $$I_t = \{ x \in I, T(x,x_0) > t\}.$$
We then say that $X$ is admissible at the target $x_0$ for the capacity $\lambda_0$ if for all $t\geq 0$, $\lambda_t$ is an admissible capacity on $I_t$, and $X_t \in \overline I_t$. 
Similarly, we can note that when $X$ is admissible at $x_0$ for the capacity $\lambda_0$, then the process $(X_{t+\cdot})$ is admissible at $x_0$ for the capacity $\lambda_t$, and that this definition does not depend on the time-parametrization of $X$.

We can actually realize $T(x,x_0)$ as the first time at which the capacity $\lambda_t$ reaches 0 at some point inside $[x,x_0)$. The process $X$ can therefore also be realized as living in the varying capacity $\lambda_t$ that it eats up continuously, on the interval $I_t$ that gets smaller every time the capacity $\lambda_t$ reaches 0 at some new point. In some sense, the capacity $\lambda_t$ left at time $t$ represents some remaining local time available for the process $X$ to ``eat out" or to exhaust, so we will sometimes take this point of view and call $\lambda_t$ the available local-time left at time $t$.

\begin{remark}
Note that for $X$ admissible for $\lambda_0$, the admissibility condition imposes $[X]_{\infty}  := [X]_{T(X)} \leq \int_I \lambda_0$, so the total quadratic variation of $X$ is bounded by the total initial available local time. This means in particular that in the case of \emph{finite} initial capacity $\int_I \lambda_0 < +\infty$, any admissible process $X$ will have finite quadratic variation and thus it will be natural to parametrize it by its quadratic variation up to full time. We call this parametrization the \emph{standard} parametrization, or the quadratic variation parametrization.
\end{remark}

\subsection{Stretching flows and Drivers}
In all this article, $\1_{(a,b)}(y)$ denotes the \textbf{signed} indicator of the oriented interval $(a,b)$, with the usual extension when an endpoint is infinite, namely
$$\1_{(a,b)}(y):=\mathbf 1_{\{a<y<b\}}-\mathbf 1_{\{b<y<a\}}.$$
Let us fix an admissible process $(X_t)_{0 \leq t  \leq T(X)}$ at the target $x_0\in \overline{I}$ for the capacity $\lambda_0$ on $I$.
For $g_0$ a stretching morphism from $\lambda_0$ to the constant capacity 1 on $\R$, as discussed in the introduction, it is natural to associate the \textit{stretching flow} $(g_t)_{0 \leq t  \leq T(X)}$ defined by the following flow of integral equations:
\begin{equation}
\label{floweq}
g_t(x)=g_0(x)+ \int_0^t\frac{\1_{(x_0,x)}(X_s)}{\lambda_s(X_s)^2}\dd [X]_s, \quad \text{for } x\in I.
\end{equation}

For any fixed $x\in I$, the solution $g_t(x)$ is defined up until the disconnection time $T(x,x_0)$, therefore $g_t$ is well-defined on the interval $I_t$. 

Notice that the stretching flow $(g_t)$ is chosen to be invariant by time-change. Indeed, for any absolutely continuous time-change $\xi$, if we set $\tilde X_t = X_{\xi(t)}$ and $\tilde \lambda_t = \lambda_0 - L_{\tilde X}(t,\cdot) = \lambda_{\xi(t)}$, then
$$g_{\xi(t)}(x)-g_0(x) =\int_0^{\xi(t)}\frac{\1_{(x_0,x)}(X_s)}{\lambda_s(X_s)^2}\dd[X]_s =\int_0^t\frac{\1_{(x_0,x)}(\tilde X_u)}{\tilde\lambda_u(\tilde X_u)^2 }\dd[\tilde X]_u.$$
Therefore the stretching flow $\tilde g$ associated to $\tilde X$ is simply $(\tilde g_t )= (g_{\xi(t)})$, the time-changed flow associated to $X$, so that this flow is well-defined even up to time-parametrization.

\ms

The motivation behind this definition is that the process $X$ ``exhausts" the available capacity $\lambda_t$ at rate $\dd [X]_t$ at the point $X_t$, so the stretching that uniformizes the exhausted capacity $\lambda_t$ to the constant capacity $1$ and sends $x_0$ to $z_0:=g_0(x_0)\in \overline\R$ has to be the initial stretching $g_0$, compensated on each side of $x_0$. The following proposition indeed states that defining $g_t$ as we did is the exact choice necessary to get the right uniformization.

\begin{proposition}[Stretching flow]
\label{stretchflow}
For all $t\geq 0$, the map $g_t : I_t \to \R$ defined by (\ref{floweq}) is a stretching morphism from the capacity $\lambda_t$ on $I_t$ to the constant capacity $1$ on $\R$, with $g_t(x_0) = g_0(x_0) = z_0$. If $x_0\in I$, it is the unique such stretching, and if $x_0\in\partial I$, equation~(\ref{floweq}) fixes the otherwise free translation through the initial uniformization $g_0$.
\end{proposition}

\begin{proof}
Fix $x,y\in I_t$. By the occupation-density formula and $\lambda_t=\lambda_0-L_t$,
\begin{align*}
R_{\lambda_t}(x,y)-R_{\lambda_0}(x,y)
&=\int_x^y\left(\frac1{\lambda_t(r)}-\frac1{\lambda_0(r)}\right)\dd r\\
&=\int_{x}^y\int_0^t \frac{\dd L_s(r)}{\lambda_s(r)^2}\dd r =\int_0^t\frac{\1_{(x,y)}(X_s)}{\lambda_s(X_s)^2}\dd[X]_s.
\end{align*}
Away from the interval endpoints, the additivity of signed interval indicators gives $\1_{(x_0,y)}-\1_{(x_0,x)}=\1_{(x,y)}.$
Consequently, equation~(\ref{floweq}) and the fact that $g_0$ stretches the capacity $\lambda_0$ to the constant capacity $1$ yield
$$g_t(y)-g_t(x)=R_{\lambda_t}(x,y).$$
Therefore $g_t$ is indeed a stretching morphism from $\lambda_t$ to the constant capacity $1$. For an interior target, evaluating (\ref{floweq}) at $x_0$ fixes $g_t(x_0)=z_0$ and hence gives uniqueness. For a boundary target, the same equation evaluated at any interior point fixes the translation relative to $g_0$.
\end{proof}

When we have such a flow of stretching morphisms $(g_t)$ associated to the process $X$, we can apply it to the process $X$ itself since we always have $X_t \in \overline I_t$. This defines the \textit{driver process} $D_t = g_t(X_t)$, taking values in $\overline\R = \R\cup\{\pm\infty\}$. In some sense, while our process $X$ lives on the \textit{varying} capacity $\lambda_t$, its driver $D$ is continuously uniformized to live on the constant capacity $1$. Here, the driver $D$ and the stretching flow $g$ depend of course heavily on the initial capacity $\lambda_0$, the target anchor $x_0$ and the initial uniformization $g_0$.

\ms

Locally, the fact that $g_t$ is a stretching morphism from $\lambda_t$ to $1$ means that for $x\in I_t$, $\partial_x g_t(x) = 1/\lambda_t(x) $. Here, the flow $(t,x)\mapsto g_t(x)$ is $\mathcal C^1$ in space and of finite variation in time, therefore the quadratic variation of the driver $ D_t = g_t(X_t)$ is explicit
$$\dd [D]_t = \frac{\dd [X]_t}{\lambda_t(X_t)^2}. $$
Since stretching morphisms preserve the order, the condition $x_0<X_t<x$ is equivalent to $ g_t(x_0) = z_0<D_t<g_t(x) $, and we can then rewrite the flow equation (\ref{floweq}) in terms of the driver :
\begin{equation}
\label{floweq2}
g_t(x)=g_0(x)+\int_0^t\1_{(z_0,g_s(x))}(D_s)\dd [D]_s.
\end{equation} 
This rewriting is once again invariant by time-change, and useful insofar as it completely determines the stretching flow $(g_t)$ only from the driver $(D_t)$, and thus also allows us to get back to the original process $(X_t)$ only from the knowledge of its driver.

\ms

We first require a technical Lemma about the solutions of the flow equation (\ref{floweq2}) in the spirit of what is done in \cite{bass1998stochasticbifurcationmodels} and \cite{HU2000287}.

\begin{lemma}
\label{lemmatech}
Let $g_0$ be a stretching morphism from some capacity $\lambda$ on an interval $I$ to the constant capacity 1, and fix $x_0\in\overline I$. This means that for $z_0:=g_0(x_0)\in\overline\R$, $g_0$ is a $\mathcal C ^1$-diffeomorphism from $I$ to $\R$ of the form $g_0 : x\in I \mapsto z_0+\int_{x_0}^x {\dd r} / \lambda(r) \in \R $ (understood as a limit when $z_0=\pm\infty$). Let $(D_t)_{0\leq t\leq T}$ be a continuous process with strictly increasing quadratic variation, taking values in $\overline \R = \R\cup\{\pm\infty\}$ (defined up to time-change) and verifying the following two conditions.

\medskip

\noindent
\textnormal{(i) Well-posedness of the flow.}
For every $x\in I$, the integral equation
$$g_t(x)=g_0(x)+\int_0^t\1_{(z_0,g_s(x))}(D_s)\dd[D]_s$$
admits a unique solution until time $T(x):=\inf\{t\geq 0, g_t(x)=\pm\infty\}$, and the map
$(t,x)\mapsto g_t(x)$
is continuous on its domain $\{(t,x), x\in I, 0\le t < T(x)\}$. 

\medskip

\noindent
\textnormal{(ii) Jointly continuous local times along the flow.}
For $x\in I$, assume that the family of processes $(R_t^x)_{0\leq t <T(x)} :=(D_t-g_t(x))_{0\leq t <T(x)} $ admits jointly continuous local
times: namely, assume that there exists a nonnegative field
$(L_t^x(a))_{x\in I,a\in\R,\, 0\leq t<T(x)}$, jointly continuous (at least in a neighborhood of $a=0$)
in $(t,x,a)$, such that for every $x\in I$, every $0\leq t<T(x)$,
and every bounded Borel function $f$,
$$\int_0^t f(R_s^x)\dd[R^x]_s=\int_{\mathbb R}f(a)L_t^x(a)\dd a.
$$

\ms

Write   
$L_t(x):=L_t^x(0).$ For $0\leq t \leq T$, let $I_t := \{ x \in I,\, T(x) > t\}$, and finally let us define for such a process $D$ its final time $T(D) := \inf\{s\leq T,\, \mathring I_s=\emptyset\}\wedge T $.
Then the following properties hold.

\begin{enumerate}
\item For every $0 \leq t < T(D)$, the map $x\mapsto g_t(x)$ is an increasing $C^1$-diffeomorphism from the open interval $I_t := \{ x \in I,\, T(x) > t\}$ to $\R$, with
$$\partial_x g_t(x)=({\exp(L_t(x))}) / {\lambda(x)}.$$

\item  The process $X$ defined for $0\leq t< T(D)$ by
$X_t:=g_t^{-1}(D_t)\in \overline I_t$
is continuous and admits a bicontinuous local time $(\Lambda_t(x))_{0\leq t <T(D),\,x\in I}$, where
$\Lambda_t(x)=\lambda(x)\left(1-\exp\bigl(-L_t(x)\bigr)\right)$, i.e.\ for every bounded Borel function $f$, 
$$\int_0^t f(X_s)\dd[X]_s=\int_{\mathbb R}f(x)\Lambda_t(x)\dd x.
$$
\end{enumerate}
\end{lemma}

\begin{proof} 
It is straightforward that for all $t>0$, the unique solution $x\mapsto g_t(x)$ is strictly increasing (and thus order-preserving). 
Fix $x\in I$. For $h>0$ and $t<T(x)$, let
\begin{eqnarray*}
q_t^h:=\frac{g_t(x+h)-g_t(x)}{h} &=& \frac{1}{h} \bigl( g_0(x+h)-g_0(x)+
\int_0^t\mathbf{1}_{\{g_s(x)< D_s < g_s(x+h)\}} \dd [D]_s \bigr) \\
&=& \frac{g_0(x+h)-g_0(x)}{h}+\frac1h\int_0^t\mathbf{1}_{\{0<R_s^x<hq_s^h\}}\dd [R^x]_s.
\end{eqnarray*}
Using the occupation-density formula and the change of variables $a=hr$, we obtain
$$q_t^h=\frac{g_0(x+h)-g_0(x)}{h}+\int_0^\infty\int_0^t\mathbf{1}_{\{r<q_s^h\}}\dd_sL_s^{x}(hr)\dd r.$$
For $0\leq t < T(x)$, let
$Q_t:=\exp\bigl(L_t(x)\bigr)/\lambda(x)$.
Since $L_0(x)=0$ and $t\mapsto L_t(x)$ is continuous and nondecreasing,
we have
\begin{equation}
\label{eq:stieltjes-exponential}
Q_t=\frac{1}{\lambda(x)} +\int_0^tQ_s\dd L_s(x).
\end{equation}
Fix $0<T<T(x)$, choose $M>Q_T$, and define $\tau_h := \inf\{t\geq0:q_t^h\geq M\}\wedge T$. We will show that for $h$ small enough, $q^h_t$ comes close to $Q_t$ and thus is smaller than $M$, so that actually $\tau_h = T$. 
By definition, for $t\leq\tau_h$, we have $q_t^h\leq M$, therefore
$$q_t^h=\frac{g_0(x+h)-g_0(x)}{h}+\int_0^M\int_0^t\mathbf{1}_{\{r<q_s^h\}}\dd_sL_s^{x}(hr)\dd r.$$
Adding and subtracting $L_s^x(0)$, we get
\begin{equation}
\label{eq:one-sided-stieltjes-approximation}
q_t^h
=
\frac{1}{\lambda(x)}+\int_0^tq_s^h\dd L_s(x)+S_t^h+\hat S^h,
\qquad 0\leq t\leq\tau_h,
\end{equation}
where
$$S_t^h:=\int_0^M\int_0^t\mathbf{1}_{\{r<q_s^h\}}\dd_s\bigl(L_s^{x}(hr)-L_s^x(0)\bigr)\dd r,$$
and, since $g_0$ is a stretching morphism from capacity $\lambda$ to the constant capacity 1,
$$ \hat S^h:=\frac{g_0(x+h)-g_0(x)}{h} -\frac{1}{\lambda(x)} \longrightarrow 0 \quad \text{as } h\to 0.$$
For every fixed $r$, the map $s\mapsto\mathbf{1}_{\{r<q_s^h\}}$
is nondecreasing and has total variation at most $1$. Therefore, we get the upper bound
$$\sup_{0\leq t\leq\tau_h}|S_t^h|\leq2M\sup_{\substack{0\leq s\leq T\\0\leq r\leq M}}\left|L_s^{x}(hr)-L_s^x(0)\right| =:\varepsilon_h$$
where, by the bicontinuity of the local-time profile $L^x$,
$\varepsilon_h\to 0$ as $h\to 0$.
Subtracting \eqref{eq:stieltjes-exponential} from
\eqref{eq:one-sided-stieltjes-approximation}, and applying Gronwall's lemma, we then obtain
$$\sup_{0\leq s\leq t} |q_s^h-Q_s| \leq (\varepsilon_h+|\hat S^h|)\exp\bigl(L_t(x)\bigr), \qquad 0\leq t\leq\tau_h.$$
For all sufficiently small $h$, we have $(\varepsilon_h+|\hat S^h|)\exp\bigl(L_T(x)\bigr)<M-Q_T.$
Therefore at time $\tau_h$,
$$q_{\tau_h}^h\leq Q_{\tau_h}+(\varepsilon_h+|\hat S^h|)\exp\bigl(L_T(x)\bigr)<Q_T+(M-Q_T)=M,$$
and necessarily
$\tau_h=T$. Consequently, as $h\to 0$, 
$$\sup_{0\leq t\leq T}\left|\frac{g_t(x+h)-g_t(x)}{h}-\frac{\exp\bigl(L_t(x)\bigr)}{\lambda(x)}\right|\longrightarrow 0.$$
We get that the map $x\mapsto g_t(x)$ is continuous with right derivative $\exp\bigl(L_t(x)\bigr)/\lambda(x)$, which is also continuous, therefore it is of class $\mathcal C^1$, with strictly positive derivative
$$\partial_x g_t(x)=\frac{\exp\bigl(L_t(x)\bigr)}{\lambda(x)}>0.$$
The fact that $g_t$ is continuous and increasing implies immediately that $I_t:= \{ x \in I,\, T(x) > t\}$ is an open interval, and imposes
$$\lim_{x\to\inf I_t}g_t(x)=-\infty,\qquad\text{and}\qquad\lim_{x\to\sup I_t}g_t(x)=+\infty.$$
Therefore $g_t$ is indeed an increasing $\mathcal C^1$-diffeomorphism from $I_t$ to $\R$.
The joint continuity of $g$ and of its inverse implies that $X_t=g_t^{-1}(D_t)$ is continuous. Up to the endpoints, which have zero occupation measure with respect to $\dd[D]$, we have
$$\1_{(z_0,g_s(x))}(D_s)=\mathbf 1_{\{D_s<g_s(x)\}}-\mathbf 1_{\{D_s<z_0\}},$$
the flow equation can be rewritten as
$$g_t(x)-g_0(x)+A_t^{z_0}=\int_0^t\mathbf{1}_{\{X_s<x\}}\dd [D]_s,
\qquad A_t^{z_0}:=\int_0^t\mathbf 1_{\{D_s<z_0\}}\dd[D]_s.$$
The left-hand side is continuously differentiable in $x$, and $A_t^{z_0}$ does not depend on $x$. Hence the occupation measure of $X$ relative to the measure $\dd [D]_t$ on $[0,t]$ is absolutely continuous, with density $\ell_t(x)$ verifying
$$\ell_t(x) = \partial_x g_t(x)-\partial_x g_0(x)=\left(\exp\bigl(L_t(x)\bigr)-1 \right)/\lambda(x).$$
Moreover, $D_t = g_t(X_t)$ and $g_t$ is of bounded variation in $t$, thus $X$ has quadratic variation verifying $\dd [D]_t = \bigl(\partial_x g_t(X_t)\bigr)^2\dd [X]_t = \lambda(X_t)^{-2}\exp\bigl(2L_t(X_t)\bigr)\dd [X]_t $. Therefore, the local time profile $\bigl(\Lambda_t(x)\bigr)$ of $X$, which is by definition its occupation measure relative to $\dd[X]_t$, is
$$ \Lambda_t(x) = \int_0^t\lambda(x)^{2}\exp\bigl(-2L_s(x)\bigr)\dd\ell_s(x)= \lambda(x)\left(1-\exp\bigl(-L_t(x)\bigr)\right).$$
The joint continuity of $\Lambda$ follows from that of $L$.
\end{proof}

 What is to be understood of this Lemma is that these are the natural conditions required to be able to consider a process $D$ until final time $T(D)$ as a driver of some other process. In particular, conditions (i) and (ii) of this technical Lemma \ref{lemmatech} are verified by any driver process $D$ of an actual admissible process $X$. Therefore, from now we call any such process $(D_t)_{0\leq t \leq T(D)}$ (defined up to time-reparametrization) verifying these conditions, an \textit{admissible driver} (or simply \textit{driver}, since this is coherent with the fact that any driver of an admissible process is an actual admissible driver). In particular, a driver process is necessarily continuous, has quadratic variation and bicontinuous local times.

\begin{proposition}[Characterization of an admissible process by its driver]
\label{characterization}
Let $\lambda_0$ be an admissible capacity on an interval $I$, let $g_0$ be a stretching morphism from $\lambda_0$ to the constant capacity $1$ on $\R$, and let $x_0\in \overline I$. Let $(D_t)_{0\leq t < T(D)}$ be an admissible driver. Then, there exists a unique process $(X_t)_{0\leq t < T(D)}$ (defined up to time-change) admissible at $x_0$ for the capacity $\lambda_0$ on $I$ such that $(D_t)_{0\leq t < T(D)}$ is the driver of the process $X$ for the target $x_0$ and initial uniformization $g_0$.

\end{proposition}

\begin{proof}
Let us proceed by analysis-synthesis. We write the proof when the target
$x_0$ is an interior point of $I$, the boundary case is obtained in the
same way by keeping only the corresponding one-sided intervals. By the discussion preceding the previous Lemma, necessarily the eventual stretching flow $(g _t )$ has to verify the flow equation (\ref{floweq2}) with $z_0 := g_0(x_0)$,
$$ g_t(x)=g_0(x)+\int_0^t\1_{(z_0,g_s(x))}(D_s)\dd [ D]_s, \quad \text{for } x\in I.$$
So as discussed, up to a time-change, the only candidate is the process $X$ defined for $0\leq t < T(D)$ by $ X_t := g_t^{-1}(D_t)$. 

Reciprocally, Lemma \ref{lemmatech} gives exactly that this process $X$ is admissible for the capacity $\lambda_0$, and that its driver is actually the admissible driver $D$. Indeed, $X$ is continuous and admits a bicontinuous local times $\Lambda_t(x)=\lambda_0(x)\left(1-\exp\bigl(-L_t(x)\bigr)\right)\leq \lambda_0(x)$. Moreover, let $\tilde g$ be the corresponding stretching flow for initial stretching $g_0$. By definition, $\tilde g$ verifies the flow equations (\ref{floweq}).
\begin{eqnarray*}
    \tilde g_t(x)&=& g_0(x)+ \int_0^t\frac{\1_{(x_0,x)}(X_s)}{\lambda_s(X_s)^2}\dd [X]_s = g_0(x)+ \int_0^t\frac{\1_{(x_0,x)}(X_s)}{(\lambda_0(X_s)-\Lambda_s(X_s))^2}\dd [X]_s \\
    &=&  g_0(x)+ \int_0^t\frac{\1_{(x_0,x)}(X_s)}{(\lambda_0(X_s)\exp\bigl(-L_s(X_s)\bigr))^2}\dd [X]_s \\
    &=&  g_0(x)+ \int_0^t\1_{(x_0,x)}(X_s)\bigl(\partial_x g_s(X_s)\bigr)^2\dd [X]_s \\
    &=& g_0(x)+ \int_0^t\1_{(g_s(x_0),g_s(x))}(D_s)\dd [D]_s\\
    &=& g_t(x).
\end{eqnarray*}
Therefore, the flow $(g_t(x))$ is indeed the stretching flow associated to $X$ and the process $D$ is its driver.
\end{proof}

This characterization allows two things. First, to define an admissible process $X$ on some capacity $\lambda_0$, we can define a driver process $D$ on $\overline\R$, then for any initial stretching $g_0$ from $\lambda_0$ to $1$ and any target $z_0\in\overline\R$, there exists a unique process $X$ whose driver for $g_0$ at $x_0:=g_0^{-1}(z_0)$ is $D$, which we call the process \textit{driven} by $D$. Second, since the flow equation (\ref{floweq2}) is invariant by time-change, the process driven by a time-changed version of $D$ will be a time-change of $X$, so it will be the same process defined up to time-reparametrization. This justifies a posteriori that we can also understand drivers up to time-reparametrization ; and we will do so in the rest of the article.

\ms

The following proposition states that if $X,Y$ are two admissible processes defined on respective capacities $\lambda,\rho$ via the same driver $D$, then $Y= g(X)$ for $g$ some stretching morphism from $\lambda$ to $\rho$.

Let $\lambda$ and $\rho$ be two admissible capacities on respective intervals $I$ and $J$. Let $g$ be a stretching morphism from $\lambda$ to $\rho$, and let $h$ be a stretching morphism from $\rho$ to the constant capacity $1$ on $\R$. Let $x_0\in \overline I$, set $f = h\circ g$, $z_0 = g(x_0)\in \overline J$ and $y_0=f(x_0)\in\overline \R$. Let $(D_t)_{0\leq t \leq T(D)}$ be an admissible driver on $\overline\R$ at $y_0$. Let $(X_t)_{0 \leq t \leq T(D)}$ be the unique (up to time-change) admissible process at $x_0$ for the capacity $\lambda$, driven by $(D_t)_{0 \leq t \leq T(D)}$ for the initial uniformization $f$, and let $(f_t)_{0\leq t \leq T(D)}$ be the corresponding stretching flow.

\begin{proposition}[Transformation of admissible processes]
\label{transfprocess}

The process $ (g(X_t))_{0\leq t \leq T(D)}$ is the unique (up to time-change) admissible process at $z_0$ for the capacity $\rho$, driven by the same driver $(D_t)_{0 \leq t \leq T(D)}$ for the initial uniformization $h$; and the flow $(h_t)_{0\leq t \leq T(D)} = (f_t \circ g^{-1})_{0\leq t \leq T(D)}$ is the corresponding stretching flow.

\end{proposition}

\begin{proof}
For $y\in J$, with $y_0 = f(x_0)=h(z_0)$, the flow equation for $(f_t)$ gives
$$f_t(g^{-1}(y))=h(y)+\int_0^t\1_{(y_0,f_s(g^{-1}(y)))}(D_s)\dd[D]_s.$$
The flow $(f_t\circ g^{-1})_{0\leq t \leq T(D)}$ is solution to the flow equations (\ref{floweq2}) driven by $D$ and with initial stretching $f\circ g^{-1} = h$, so by Proposition \ref{characterization} it is the unique such stretching flow, and the unique corresponding admissible process driven by $D$ is then $Y_t = (f_t\circ g^{-1})^{-1}(D_t)= g(X_t)$.
\end{proof}

\subsection{The Brownian burglars and their Drivers}
The previous discussions and in particular Proposition \ref{characterization} explain how to define an admissible process on any admissible capacity $\lambda$ via its driver. An important family of admissible drivers for our work will be the following.

\begin{definition}[The burglar Drivers] \label{def:burglar_driver}
Let $\nu\in\R$, and let $z,z_0\in\overline\R=\R\cup\{\pm\infty\}$. The \emph{burglar Driver} of parameter $\nu$, starting from $z$ and targeting $z_0$, is the continuous diffusion on $\overline \R$ which satisfies $D_0 = z$, evolves according to the SDE
\begin{equation}
\label{eq:fastdriver}
\dd D_t = \exp\left(\cosh(D_t)\right)\dd B_t+ \exp\left(2\cosh(D_t)\right)(1-\nu)\sign(D_t-z_0)\dd t.
\end{equation}
for $B$ a usual Brownian motion, and is instantly reflected at the boundaries $\pm\infty$. In the case of a boundary target $z_0=\pm\infty$, we stop it at the time $T(D):=\inf\{t\geq 0, D_t = z_0\}$, and otherwise let it run until $T(D)=+\infty$, which corresponds to the final time of the driver $D$ for the target $z_0$ as defined in Lemma \ref{lemmatech}.
\end{definition}

Notice that this SDE is a time-change of the following
\begin{equation}\label{eq:uniform-driver}
\dd D_t=\dd B_t+(1-\nu)\sign(D_t-z_0)\dd t.
\end{equation}
The point of this time-change is that for the SDE (\ref{eq:fastdriver}), the Feller classification of the boundaries is the following : 
when $z_0\in\R$, the boundaries $\pm\infty$ are Feller-regular for all $\nu<1$, and are entrance boundaries for all $\nu\geq 1$ ; and if $z_0=\pm\infty$ is a boundary target, then it is an entrance boundary for $\nu\leq 1$ and a regular boundary for $\nu >1$ (while the other boundary still has the same classification as in the interior-target case). Therefore, the SDE~(\ref{eq:fastdriver}) indeed always has a unique strong solution reflected at the boundaries $\pm\infty$.
Since we consider the drivers up to time reparametrization, it is then natural to define them on this faster time-scale.
This can be summarized by the following table : 

\begin{table}[h]
    \centering
    \begin{tabular}{lcc}
    \hline
    \textbf{$\nu\in\R$} & \textbf{Boundary $\pm\infty$ different from $z_0$} & \textbf{Boundary target $z_0=\pm\infty$} \\
    \hline
    $\nu > 1$ & Entrance & Regular \\
    $\nu = 1$ & Entrance & Entrance \\
    $\nu < 1$ & Regular & Entrance \\
    \hline
    \end{tabular}
    \caption{Feller classification of the boundaries $\pm\infty$ as a function of $\nu$}
    \label{tab:feller_classification}
    \end{table}

In the SDE (\ref{eq:uniform-driver}), the drivers of $\nu<1$ almost surely diverge to $\pm \infty$, and with this parametrization they take infinite time to do so. As a consequence, it was already noticed in \cite{AHS,Lupu_2021} that with such a parametrization, the corresponding burglar processes can only be defined up to an almost surely finite stopping time, while the SDE (\ref{eq:fastdriver}) allows the driver path to be well-defined for a longer ``intrinsic time". We will even show later that in some sense, Definition \ref{def:burglar_driver} is actually maximal in time, i.e. these drivers define an admissible process that could not last any longer than this.

\ms

Since the drivers are defined up to time-reparametrization, for simplicity of the calculations in the rest of the paper we will work with the time-parametrization of (\ref{eq:uniform-driver}), but it is straightforward that all the corresponding results and definitions also hold in the time-prametrization of (\ref{eq:fastdriver}).
\ms

As suggested by their name, the burglar drivers are destined to be understood as driver processes. In particular, they are indeed admissible drivers by the work of \cite{bass1998stochasticbifurcationmodels, HU2000287}, and their final time $T(D)$ is always $+\infty$ except in the case explicitly mentioned in  Definition \ref{def:burglar_driver}.

\ms

We now define the central objects of the article, the general Brownian burglars, via their driver processes thanks to Proposition \ref{characterization}.

\begin{definition}[The Brownian burglar] \label{def:general_burglar}
Let $\lambda_0$ be an admissible capacity on an interval $I$, and let $x,x_0\in\overline I$. Let
$g$ be any stretching morphism from $\lambda_0$ to the constant capacity $1$, and set $z_0=g(x_0)\in\overline\R$. The Brownian burglar $X$ of parameter $\nu$ on $\lambda_0$, starting from $x$ and targeting $x_0$, is defined (up to time-change) as the unique admissible process at $x_0$ for $\lambda_0$, driven by $D$ the burglar Driver of parameter $\nu$ started from $g(x)$ and targeting $g(x_0)$, for the initial stretching $g$.
\end{definition}

In order to verify that the Brownian burglar is well-defined, we need to check that its law does not depend on the choice of initial stretching morphism $g$. Any other stretching morphism from $\lambda_0$ to the constant capacity is of the form $\hat g=g+c$, for $c\in\R$. Up to time-change, the corresponding flow and driver are then $\hat g_t=g_t+c$ and $\hat D_t=D_t+c$. Equation~(\ref{eq:uniform-driver}) is translation invariant and
$$\hat g_t^{-1}(\hat D_t)=(g_t+c)^{-1}(D_t+c)=g_t^{-1}(D_t),$$
therefore the law of the burglar is well-defined.

\begin{remark}
  In the case $x=x_0\in\partial I$, we have defined the corresponding burglar driver to be trivial (it has final time 0), and thus the burglar is also trivial and killed instantly. For $\nu>0$, one can verify by choosing instead a starting point $y\neq x_0$ and doing $y\to x_0$ that the limiting process is indeed trivial. However, for $\nu \leq 0$, it is not the case, so that one could actually define a non-trivial burglar targeting $x_0\in\partial I$ and started from this target $x_0$, by taking the limit as $y\to x_0$. Another way to define this could be to consider the stretching flow (\ref{floweq2}) centered instead at another point $z$ that $z_0=\pm\infty$ in this case, which would allow a definition until $z$ gets swallowed, and then letting $z\to z_0$.
\end{remark}

\section{Properties and characterization of the burglars}
\label{properties}

We are going to give several important properties of the burglars defined through these drivers. We can already find some restricted versions of some of these properties in \cite{AHS} via the definition through disintegration of a Brownian motion or via a Bass-Burdzy flow, and it is quite easy to adapt most of the proofs to our setup.

\subsection{Stretching covariance}

Let $\lambda$ be an admissible capacity on an interval $I$, and let $x,x_0 \in \overline{I}$. Let $X$ be a Brownian burglar of intensity $\nu \in\R$ and initial capacity $\lambda$, starting from $x$ and targeting $x_0$. Let $g$ be a stretching morphism from $\lambda$ to an other capacity $\rho$ on some interval $J$.

\begin{proposition}[Stretching covariance]
The process $(g(X_t))_{t\geq 0}$ is (up to time-change) a Brownian burglar of intensity $\nu$ and initial capacity $\rho$, starting from $g(x)$ and targeting $g(x_0)$.
\end{proposition}

\begin{proof}
This is exactly the application of Proposition \ref{transfprocess} to the Brownian burglars and their drivers.
\end{proof}

\subsection{Symmetry}

It is direct from (\ref{eq:uniform-driver}) that if the burglar driver $D$ targets $z_0$, then $-D$ is a burglar driver targeting $-z_0$. By characterization of processes by their drivers, we obtain the following symmetry property for the corresponding burglars :

\begin{proposition}[Symmetry of the burglars laws]
Let $\lambda$ be an admissible capacity on an interval $I$, and let $\tilde \lambda$ be the symmetric capacity, i.e.\ $\tilde \lambda (x)= \lambda(-x)$ for $x\in-I$. Let $x,x_0 \in \overline{I}$. Let $X$ be a Brownian burglar of intensity $\nu \in\R$ and initial capacity $\lambda$, starting from $x$ and targeting $x_0$. Then $(-X_t)_{t\geq 0}$ has the law of a Brownian burglar of intensity $\nu$, initial capacity $\tilde \lambda$, starting from $-x$ and targeting $-x_0$.
\end{proposition}

\subsection{Local Time Strong Markov property}

Let $\lambda_0$ be an admissible capacity on an interval $ I$, and let $x_0 \in \overline{I}$. Let $(X_t)_{0 \leq t \leq T(X)}$ be a Brownian burglar of intensity $\nu \in\R$ and initial capacity $\lambda_0$, targeting $x_0$.

\begin{theorem}[Local-Time Strong Markov Property] \label{thm:LT_Markov}
Let $T$ be a stopping time for the natural filtration of $X$. Conditionally on $X_T$ and the remaining available capacity $\lambda_T$, the process $(X_{T+t})_{0 \leq t \leq T(X)-T}$ has the law of a Brownian burglar of intensity $\nu$ with initial capacity $\lambda_T$ on the interval $I_T$, starting from $X_T$ and targeting $x_0$, independently from $\F_T$.
\end{theorem}

\begin{proof}
Let $g_0$ be a stretching morphism from $\lambda_0$ to the capacity $1$, set $z_0=g_0(x_0)$, let $(D_t)_{0 \leq t \leq T(X)}$ be the corresponding driver of $X$ at $x_0$, and let $(g_t)_{0 \leq t \leq T(X)}$ be its stretching flow. By Proposition \ref{characterization}, $D$ and $X$ determine one another, so their natural filtrations coincide. The process $(X_{T+t})_{0 \leq t \leq T(X)-T}$ is admissible at $x_0$ for $\lambda_T$ on $I_T$. Its stretching flow started from $g_T$ is $(g_{T+t})_{0 \leq t \leq T(X)-T}$, and its driver is $(D_{T+t})_{0 \leq t \leq T(X)-T}$. Since $D$ is defined via (\ref{eq:fastdriver}), it verifies the strong Markov property, and conditionally on $D_T$, $(D_{T+t})_{0 \leq t \leq T(X)-T}$ is a burglar driver of parameter $\nu$ started from $D_T$ and targeting $z_0 = g_T(x_0)$, independently of $\F_T$. Therefore, $X_{T+t} = g_{T+t}^{-1}(D_{T+t})$ is indeed a Brownian burglar of intensity $\nu$ with initial capacity $\lambda_T$ on the interval $I_T$, started from $X_T$ and targeting $x_0$, independently of $\F_T$ conditionally on $X_T$ and $\lambda_T$.
\end{proof}

\subsection{Side-locality}
\label{sidelocality}

Let $\lambda_0$ be an admissible capacity on an interval $I$, $x < x_1 < x_2 \in \overline I$.
Let $X^{(1)}$ and $X^{(2)}$ be two burglars of some intensity $\nu \in \R$, starting at $x$ with initial capacity $\lambda_0$ and targeting respectively $x_1$ and $x_2$. Let $\tau^{(i)}_{x_1} = \inf\{t\geq 0, X^{(i)}_t = x_1\}$, and let $T^{(i)}(x_1, x_2)$ be the usual disconnection time of $x_1$ and $x_2$ by the process $X^{(i)}$.

\begin{proposition}
\label{sideloc}
The stopped processes $\left( X^{(1)}_t \right)_{0 \le t \le \tau^{(1)}_{x_1}} $ and $ \left( X^{(2)}_t \right)_{0 \le t \le \tau^{(2)}_{x_1}}$ have same law, and the two burglars $X^{(1)}$ and $X^{(2)}$ can therefore be coupled before reaching $x_1$.
\end{proposition}

\begin{proof}
Let $g$ be a stretching morphism from $\lambda_0$ to the constant capacity $1$, and $z_1 =g(x_1)$, $z_2=g(x_2)$.
By definition, the driver $D^{(2)}$ (resp $D^{(1)}$) of $X^{(2)}$ (resp $X^{(1)}$) at $x_2$ (resp $x_1$), with initial stretching $g$ is the burglar driver of parameter $\nu$ at $z_2$ (resp $z_1$). Let us compute $D'^{(2)}$, the driver of $X^{(2)}$ with initial stretching $g$, but at the target $x_1$ instead (at which the process $X^{(2)}$ is admissible until the time $T^{(2)}(x_1, x_2)$). Let $(g_t^{(2)})$ be the stretching flow of $X^{(2)}$ at $x_2$ and $(g_t'^{(2)})$ at $x_1$. Then $g_t'^{(2)}\circ\left(g_t^{(2)}\right)^{-1}$ is the unique Möbius transform of the constant capacity $1$ (thus a translation) sending $Y_t := g_t^{(2)}(x_1)$ to $g_t'^{(2)}(x_1) = z_1$, therefore
$$ D'^{(2)}_t = g_t'^{(2)}(X^{(2)}_t) =  g_t'^{(2)}\left(\left(g_t^{(2)}\right)^{-1}(D^{(2)}_t)\right) = D^{(2)}_t +z_1 - Y_t.$$ 
By definition (equation~(\ref{floweq2})) and since $x_1<x_2$ and $g_t$ preserve the order, $Y_t<z_2$ and
$$\dd Y_t=\1_{(z_2,Y_t)}(D^{(2)}_t)\dd [D^{(2)}]_t = -\mathbf 1_{(Y_t,z_2)}(D^{(2)}_t)\dd [D^{(2)}]_t.$$
Combining this identity with the SDE of the burglar drivers (\ref{eq:uniform-driver}) gives, in the quadratic variation parametrization of the driver that we use for calculations (as discussed at the end of Section \ref{formalism})
$$\dd D_t'^{(2)}=\dd B_t+b(D_t'^{(2)},Z_t)\dd t,$$
where $Z_t := g_t'^{(2)}(x_2)$ and
\begin{equation}
\label{shiftloc}
b(D, Z) = \begin{cases} 
1-\nu & \text{if } D>Z,\\
\nu & \text{if } z_1<D<Z,\\
\nu-1 & \text{if } D<z_1.
\end{cases} 
\end{equation} 
Before $X^{(2)}$ reaches $x_1$, we have $X^{(2)} < x_1$ and thus $D_t'^{(2)}<z_1$. Its driver therefore has drift $\nu-1$, so it has exactly the same law as the burglar driver of $X^{(1)}$ below its target $z_1$. The characterization of processes by their drivers allows us to conclude.
\end{proof}

The ``change of target" drift (\ref{shiftloc}) implies a duality between the burglars of parameter $\nu$ and $1-\nu$, that can be stated in the following way via the side-locality : we have proved that at any point in space, the law of the burglar of parameter $\nu$ and initial capacity $\lambda_0$ only depends on which side of its target it is. Then, locally, a burglar of parameter $1-\nu$ with target on some given side has the law of a burglar of parameter $\nu$ with target on the opposite side.
The side-locality and this duality both echo and validate heuristically the formal SDE (\ref{heursde}).

At the critical intensity $\nu=1/2$ (which is the one at which Le Jan isomorphism states that the occupation time of the corresponding loop-soup has the law of a squared Brownian motion), since $\nu = 1-\nu$, a particularly nice behaviour that we shall now briefly discuss appears.

\subsection{Target-independence at critical intensity}

Notice in the previous proof that the intensity $\nu=1/2$ plays a special role: the two drifts in (\ref{shiftloc}) above $z_1$, the image of the new target $x_1$, are both equal at this critical intensity. The dependency on the image $Z_t$ of the old target $x_2$ completely vanishes, and $D'^{(2)}$ is exactly a burglar Driver of parameter $1/2$ targeting $z_1$ until the disconnection time $T(x_1,x_2)$. By the characterization by the driver, the burglars $X^{(1)}$ and $X^{(2)}$ have the same law until that time. Thus their local behaviour is target-independent at critical intensity.
This yields the following stronger target-independence proposition for the critical intensity $\nu=1/2$, allowing us to couple burglars until the later time $T(x_1,x_2)$.

\begin{proposition}[Target-independence at $\nu = 1/2$] \label{thm:locality}
At the critical intensity $\nu=1/2$, the laws of two burglars $X^{(1)}$ and $X^{(2)}$ targeting respectively $x_1$ and $x_2$, with same initial capacity $\lambda$ and starting point $x$,  are identical until their disconnection time $T(x_1, x_2)$ :
$$\left( X^{(1)}_t \right)_{0 \le t \le T^{(1)}(x_1, x_2)} \overset{(d)}{=} \left( X^{(2)}_t \right)_{0 \le t \le T^{(2)}(x_1, x_2)}.$$
\end{proposition}

An immediate but useful consequence of this fact is that at this critical intensity, it is possible to couple two burglars targeting two different points until the exploration exhausts all local time available at some point between the two targets. This is yet another instance of the particularly tractable behaviour of the cable-graph loop-soup at critical intensity already enlightened by Le Jan isomorphism \cite{LeJan2011}, Lupu's explicit computations for connection probabilities \cite{Lupu2016LoopClustersInterlacementFreeField} and Werner's recent switching property \cite{werner2025switchingidentitycablegraphloop}.

\subsection{Link with the (skewed) Bass-Burdzy flow}
\label{bassburdzy}

Let $\beta_1$, $\beta_2\in \R$, $\sigma>0$ and $\gamma$ be a standard Brownian motion. Following \cite{bass1998stochasticbifurcationmodels} and the notations of \cite{AHS}, one defines the Bass-Burdzy flow of parameters $(\beta_1,\beta_2)$ and diffusivity $\sigma$ as the collection of homeomorphisms of the real line $({\mathcal R}_t,\,t\ge 0)$ such that for any $x\in \R$, the process $({\mathcal R}_t(x),\, t\ge 0)$ is the strong solution of the SDE
$$
{\mathcal R}_t(x) := x+ \sigma \gamma_t + \beta_1\int_0^t \mathbf 1_{\{{\mathcal R}_s(x)< 0\}} \dd s +\beta_2\int_0^t \mathbf 1_{\{{\mathcal R}_s(x)> 0\}} \dd s.
$$
When $\sigma=1$, we just call it Bass--Burdzy flow with parameters $(\beta_1,\beta_2)$.

\ms

Warren first showed in \cite{MR2197114} that this Bass-Burdzy flow arise from the study of the local times of the burglar defined in \cite{SPS_1998__32__328_0} (corresponding to the burglars of intensity $\nu=0$ in our work). Let us explain how this flow appears in the general case.

\ms

Let $X$ be a burglar of parameter $\nu$ on the constant capacity $1$ on $\R$, started from 0 and targeting $x_0\in\overline \R= \R \cup \{\pm \infty \}$. Let $(g_t)_{0 \leq t \leq T(X)}$ be the stretching flow of $X$ at the target $x_0$ started from $g_0 = \text{id}_\R$, and $(D_t)_{0 \leq t \leq T(X)}$ the corresponding driver (so that $D_t$ takes values in $\overline \R$, as detailed at the end of Section \ref{formalism}).

Let us parametrize everything to the quadratic variation of the driver $D$, i.e. consider $\xi(t) = \inf\{s\geq 0, [D]_s > t\} $. Note that here it is possible that the quadratic variation $[D]_t$ diverges in finite time (and it will indeed almost surely be the case if and only if $\nu<1$), such that in this case $\xi(t) \to t_{\text{max}} < \infty$ as $t\to \infty$. Therefore, with this time-parametrization, we lose all the information on the geometric paths of the drivers and processes, and on the flows after time $t_{\text{max}}$, and we can only define the processes until this almost surely finite stopping time. Set $\tilde g_t = g_{\xi(t)}$, $\tilde X_t = X_{\xi(t)}$, and $\tilde D_t = D_{\xi(t)}$.

Let us consider the tip-centered stretching flow $(R_t(x))_{t \geq 0, x \in \R} := (\tilde g_t(x) - \tilde D_t)_{t \geq 0, x \in \R}$. Defined this way, $R_t$ is the unique stretching morphism from the capacity $\tilde \lambda_t$ to the constant capacity 1, sending the point $\tilde X_t$ to 0. Then, this flow $(R_t(x))_{t \geq 0, x \in \R} $ is a ``skewed Bass-Burdzy flow'' of parameters $(\nu-1,\nu)$ above $x_0$ and $(-\nu,1-\nu)$ below. 
Indeed, by (\ref{eq:uniform-driver}), the burglar driver in its quadratic variation parametrization satisfies
$$ \dd \tilde D_t = \dd \tilde\gamma_t + (1-\nu)\sign(\tilde D_t-x_0) \dd t $$
for some Brownian motion $\tilde\gamma$. Moreover, the uniform flow equation (\ref{floweq2}) reads
$$ \tilde g_t(x)=g_0(x)+\int_0^t\1_{(x_0,\tilde g_s(x))}(\tilde D_s)\dd [\tilde D]_s = x+\int_0^t\1_{(x_0-\tilde D_s,R_s(x))}(0) \dd t. $$
Therefore, the flow $(R_t(x))$ verifies : 
\begin{eqnarray*}
R_t(x) = \tilde g_t(x) - \tilde D_t &=& x - \tilde \gamma_t + \int_0^t \left[ (\nu-1) + \mathbf 1_{\{R_s(x)>0\}} \right] \mathbf 1_{\{\tilde D_s>x_0\}} \dd s \\
&& - \int_0^t \left[ (\nu-1) + \mathbf 1_{\{R_s(x)<0\}} \right] \mathbf 1_{\{\tilde D_s<x_0\}} \dd s
\end{eqnarray*}
Notice that since $\tilde D_t = x_0-R_t(x_0)$, so this can also be written in the following way :

\begin{proposition}
The tip-centered stretching flow $(R_t(x))_{t \geq 0, x \in \R} = (\tilde g_t(x)-\tilde D_t)_{t \geq 0, x \in \R} $ is a ``skewed Bass-Burdzy flow'' at $x_0$, of parameters $(\nu-1,\nu)$ and $(-\nu,1-\nu)$, i.e.\
\begin{align*}
R_t(x) &= x - \gamma_t + (\nu-1)\int_0^t \mathbf 1_{\{R_s(x)<0,R_s(x_0)<0\}}\dd s + \nu\int_0^t \mathbf 1_{\{R_s(x)>0,R_s(x_0)<0\}}\dd s\\
&\quad -\nu\int_0^t \mathbf 1_{\{R_s(x)<0,R_s(x_0)>0\}}\dd s + (1-\nu)\int_0^t \mathbf 1_{\{R_s(x)>0,R_s(x_0)>0\}}\dd s.
\end{align*}
\end{proposition}

In particular, in the special case of a boundary target $x_0$, $R_t(x)$ is exactly an usual Bass-Burdzy flow with parameters $(\nu-1,\nu)$ (for $x_0=-\infty$) and $(-\nu, 1-\nu)$ (for $x_0 = + \infty$). 

This ``skewed Bass-Burdzy flow" already appears in an analogous form, restricted to the positive half-line, in Section 7.3 of \cite{AHS}.

\begin{remark}
\label{remarkcritical}
We can once again notice that at the critical intensity $\nu = 1/2$, as discussed in the previous subsection, the dependence in the target $x_0$ vanishes because the burglars at this particular intensity are target-independent and all have same law until the first disconnection time.
\end{remark}

The Bass-Burdzy construction is limited in the sense that it allows a reconstruction of the burglar only up until the time $t_{\text{max}}$, at which it exhausts the capacity at some point (this stopping time is almost surely finite if and only if $\nu<1$). Once again, this is why we have to choose a particular time-scale for the burglar's drivers of parameter $\nu<1$, so that the drivers can reach the regular boundaries in finite time and be reflected.

\subsection{Characterization of the Brownian burglars}

Let $(\Pb^\lambda_{x \to x_0})_{\lambda,x,x_0}$ denote a family of probability laws on admissible processes at $x_0$ for the capacity $\lambda$, started at $x$. This family is indexed by an admissible capacity $\lambda$ on an interval $I$, a starting point $x \in I$, and a target anchor $x_0 \in \overline I$. As usual, we consider the geometric paths of the processes, so all the corresponding processes are defined up to time-change.

\begin{definition}[Admissible Targeted Local-Time Markov Family] \label{def:axioms}
The family of laws $(\Pb^\lambda_{x \to x_0})$ is \emph{regular} if it satisfies the following conditions :
\begin{enumerate}
\item \textbf{Stretching Covariance:} If $g$ is a stretching morphism from $\lambda$ to $g(\lambda)$, then $g$ pushes forward $\Pb^\lambda_{x \to x_0}$ to $\Pb^{g(\lambda)}_{g(x) \to g(x_0)}$, i.e.\ if $X\sim \Pb^\lambda_{x \to x_0}$, then $g(X)\sim \Pb^{g(\lambda)}_{g(x) \to g(x_0)}$.
\item \textbf{Symmetry:} If $X\sim \Pb^\lambda_{x \to x_0}$, then $-X\sim \Pb^{\tilde \lambda}_{-x \to -x_0}$, where $\tilde \lambda$ is the symmetric of the capacity $\lambda$ on the interval $I$, i.e.\ $\tilde \lambda(x) = \lambda(-x)$ for $x\in -I$.
\item \textbf{Local-Time Markov Property:} If $(X_t)_{0 \le t \le T(X)} \sim \Pb^\lambda_{x \to x_0}$, then for any finite stopping time $T$ for the filtration of $X$, conditionally on $X_T$ and $\lambda_T$, we have $(X_{T+t})_{0\le t \le T(X)-T}\sim\Pb^{\lambda_T}_{X_T \to x_0}$ independently on $\F_T$.
\item \textbf{Side-Locality:} If $X\sim \Pb^\lambda_{x \to x_1}$ and $X'\sim \Pb^\lambda_{x \to x_2}$ with $x<x_1<x_2$, then $X$ and $X'$ have same law until reaching $x_1$, i.e.\ in the same time-parametrization, $$\left( X_t \right)_{0 \le t \le \tau_{x_1}}  \overset{(d)}{=}  \left( X'_t \right)_{0 \le t \le \tau'_{x_1}}.$$

\end{enumerate}
\end{definition}

\begin{remark}
Stretching Covariance and Local Time Strong Markov Property can actually be combined into one condition (analogous to the conformal Markov property of SLEs), namely a \emph{Stretching Markov property}.
For any almost surely finite stopping time $T$, let $g$ be a stretching morphism from $\lambda_T$ to a capacity $g(\lambda_T)$. Then the process $(g(X_{T+u}))_{u \ge 0}$ has the law $\Pb^{g(\lambda_T)}_{g(X_T) \to g(x_0)}$ and is independent of $\F_T$ conditionally on $X_T$ and $\lambda_T$.
In particular, $(g_T(X_{T+u}))_{u \ge 0}$ has the law of a Brownian burglar on the constant capacity $1$, starting from $g_T(X_T) = D_T$ and targeting $z_0=g_T(x_0)$, and is independent of $(D_t)_{0 \le t \le T}$ (and thus of $(X_t)_{0 \le t \le T}$) conditionally on $D_T$.

\end{remark}

All these conditions verified together characterize the family of laws $\mathbb P$ as that of the Brownian burglars, up to the real parameter $\nu\in\R$.

\begin{theorem}[Characterization of burglars] \label{thm:classification}
Let $(\Pb^\lambda_{x \to x_0})$ be a \emph{regular} family of laws, \emph{maximal in time}. Then there exists $\nu \in \R$ such that for all admissible $\lambda$ and corresponding points $x, x_0$, the law $\Pb^\lambda_{x \to x_0}$ is that of the Brownian burglar of intensity $\nu$, initial capacity $\lambda$, started from $x$ and targeting $x_0$.
\end{theorem}

\begin{remark}
By ``maximal in time", we mean that if we can couple the family of laws on paths $\Pb$ to an other regular family of laws $\Pb'$ until final time of either path, then the final times of the paths of $\Pb$ are always greater than those of $\Pb'$.

Recall once again that we choose to work with \emph{geometric paths} (defined up to time-reparametrization) rather than directly with processes, thus the tempting counterexample defined as $X_t = x$ for $0\leq t <+\infty$ can actually be reparametrized as the process with final time $0$ defined by $X_0=x$ and $T(X)=0$, so it is indeed an admissible family of processes verifying the regular conditions, but it is not at all maximal in time !
\end{remark}

\begin{proof}[Proof of Theorem]
    Let $(\Pb^\lambda_{x \to x_0})$ be such a regular family. The Stretching covariance of the family and the transformation of processes Proposition \ref{transfprocess} implies we can define the corresponding family of drivers law $(\mathbb Q_{z \to z_0})$, where for each $\lambda$ admissible capacity on an interval $I$, for $x,x_0 \in \overline I$ and $g$ Stretching morphism from $\lambda$ to the constant capacity 1, the law $\mathbb Q_{g(x) \to g(x_0)}$ is that of the driver of the process $X\sim  \Pb^\lambda_{x \to x_0}$ at $x_0$ for the initial stretching $g$. Recall that the drivers are always defined up to time-reparametrization.

    By Proposition \ref{characterization}, it suffices to characterize the family of laws $(\mathbb Q_{z \to z_0})$. We can translate the properties verified by the paths of the family $\mathbb P$ to their drivers with law $\mathbb Q$. The stretching covariance implies that $\mathbb Q$ is translation covariant (i.e.\ if $D\sim \mathbb Q_{z \to z_0}$, then for $c\in \R$, $D+c\sim \mathbb Q_{z+c \to z_0+c}$) ; it is also symmetric and side-local ; and the local-time Markov property translates in the following way : for 
    $(D_t)_{t\geq 0}\sim \mathbb Q_{z \to z_0}$, then conditionally on $D_T$, 
    $(D_{T+t})_{t\geq 0}\sim \mathbb Q_{D_T \to z_0}$ independently on $\F_T$.
    
    Let us fix $z\in \R$, and start by identifying the law $\mathbb Q_{z \to -\infty}$. A driver $D$ with this law takes values in $\overline \R$, and for every $T\geq0$ such that $D_T\in \R$, by the translation covariance together with the Markov property of the drivers, $(D_{T+t}-D_T)_{t\geq 0}\sim \mathbb Q_{0 \to -\infty}$ and is independent of $\F_T$ the filtration of $D$ until time $T$. In particular, this means that $\tilde D$ the quadratic-variation parametrized version of $D$ is a continuous Levy process, so it is of the form $\tilde D_t = z + \tilde\gamma_t+\beta t$ for some Brownian motion $\tilde\gamma$ and some $\beta\in\R$. Recall that the quadratic variation time-parametrization of $D$ allows to parametrize the path of $D$ between two times at which it is at a boundary point $\pm\infty$, and diverge as $D\to \pm\infty$, so this completely determines the law of $D$ on every excursions away from $\{\pm\infty\}$.
    
    Side-locality and symmetry allow us to identify for general $z_0$ the law of $D\sim \mathbb Q_{z \to z_0}$ on excursions away from $\{z_0\}$ and the boundary points $\{\pm \infty\}$ as being, up to reparametrization, solution of the SDE $\dd \tilde D_t = \dd \tilde\gamma_t+\beta\sign(\tilde D_t-z_0) \dd t$. The local-time Markov property imposes that the parameter $\beta$ is the same on all such excursions of the driver $D$. Changing to the time-parametrization (\ref{eq:fastdriver}) that allows a full-time reconstruction of the burglars of any intensity, we obtain that $D$ is solution of the SDE
    $$\dd D_t = \exp\left(\cosh(D_t)\right)\dd\gamma_t+ \exp\left(2\cosh(D_t)\right)\beta\sign(D_t-z_0)\dd t,$$
    on $\R\setminus\{z_0\}$, for some Brownian motion $\gamma$. Moreover, by definition the driver $D$ admits a bicontinuous local time process, in particular it is spatially continuous at the point $z_0$ (ruling out a drift term supported on $\{z_0\}$). In the chosen time-parametrization, a sticky behaviour at any point $\{z_0,\pm\infty\}$ is excluded. Finally, the maximality in time rules out any killing at any non-target points (as soon as the boundary is regular).
    There is therefore strong uniqueness of the solution to the corresponding SDE on $\overline \R$, so $D$ is the actual burglar driver of parameter $\nu = 1-\beta$.
    
    We can thus finally conclude that the law $\Pb^\lambda_{x \to x_0}$ is that of the Brownian burglar of intensity $\nu$, on the capacity $\lambda$, started from $x$ and targeting $x_0$.
    \end{proof}

\subsection{Phases of the burglars}

We keep on drawing the parallel with SLEs by identifying different ``phases" behaviour of the burglars. Let $\lambda_0$ be an admissible capacity on a bounded interval $I$, and
let $X$ be a burglar of parameter $\nu$ with initial capacity $\lambda_0$, started from $x$ and targeting $x_0$, and $L_X$ its local time process. We will describe the profile $L_X$ at time $T(X)$, the terminal time of the process $X$.

\begin{proposition}[Phases of the burglars]
\label{phases}
The burglar $X$ behaves in the following way depending on the parameter $\nu$ :

- For $\nu \geq 1 $, $X$ does not hit full local time anywhere else than its target, i.e.\ $L_X(T(X),x)<\lambda_0(x)$ for all $x\in I\setminus\{x_0\}$ almost surely.

- For $\nu \in (0,1)$, for any deterministic $x\in I\setminus\{x_0\}$, $L_X(T(X),x)<\lambda_0(x)$ almost surely, but there almost surely exists exceptional $x\in I\setminus\{x_0\}$ such that $L_X(T(X),x)=\lambda_0(x)$. 

- For $\nu \leq  0$, $X$ is ``local-time-filling", i.e.\ for all $x\in I$, $L_X(T(X),x)=\lambda_0(x)$ almost surely.

\end{proposition}

\begin{proof}

Let $g_0$ be a stretching morphism from $\lambda_0$ to the constant capacity $1$, let $g_t$ be the corresponding stretching flow from $\lambda_t$ to $1$, and let $D_t=g_t(X_t)$ be the corresponding driver. For any time-parametrization and any $t\geq0$, $D_t\in\{-\infty,+\infty\}$ if and only if $\lambda_t(X_t)=0$.

Recall that the burglar driver $D$ of parameter $\nu\in\R$ reaches the (non-target) boundaries $\{\pm\infty\}$ if and only if $\nu<1$. This gives the existence of some points $x\in I$ such that $\lambda_{T(X)}(x) = 0$ if and only if $\nu <1$.

When $\lambda_t$ reaches 0 at some point $x$, it disconnects all the points on the other side of $x$ that $x_0$. To determine if when some point $x$ is disconnected from the target $x_0$, all the available local time at $x$ has been exhausted or not, one can study the tip-centered stretching flow as in Subsection \ref{bassburdzy}. We only treat the case of a boundary target $x_0 = \inf I$ as the general case is very similar. Let $(R_t)$ be the flow of stretching morphisms from the capacity $\lambda_t$ to the constant capacity $1$, sending the point $X_t$ to the point $0$. For $T_1>0$, $R$ is well-defined until the greatest time $t_{\text{max}}$ such that for any $t\in (T_1,t_{\text{max}})$, the driver $D_t$ is not a boundary point, i.e.\ $\lambda_t(X_t)>0$. We know that in the case of boundary target $R_t(x_0)=-\infty$, so by Subsection \ref{bassburdzy}, conditionally on $R_{T_1}$, if we reparametrize the time on $(T_1, t_{\text{max}})$ to the quadratic variation of the driver via a time-change $\xi$, then $\xi(0) = T_1$, $\xi(t)\to t_{\text{max}}$ as $t\to\infty$, and for $R'_t = R_{\xi(t)}$ :
$$R'_{t}(x) = R'_{0}(x) - \gamma_t + (\nu-1)\int_0^t \mathbbm{1}(R'_s(x)<0)\dd s + \nu\int_0^t \mathbbm{1}(R'_s(x)>0)\dd s,$$
for some Brownian motion $\gamma$.
Therefore almost surely, for all $x\in\R$, $R'_t(x)\to -\infty$ as $t\to \infty$ if and only if $\nu \leq 0$. 
Going back to the process $X$, a given point $x\in I$ is disconnected from $x_0=\inf I$ at time $t_{\text{max}}$ with $\lambda_{t_{\text{max}}}(x)> 0$ if and only if $X_{t_{\text{max}}}\neq x$, i.e.\ $x_0< X_{t_{\text{max}}} < x$. In terms of the reparametrized tip-centered stretching flow $R'$, this is equivalent to $R'_t(x) \to + \infty$ as $t \to \infty$, which can happen if and only if $\nu>0$.
\end{proof}

One can understand these phases via the disintegrations that we state in the following section, giving an interpretation of the different burglars.

\section{Burglars and disintegration of loop-soups}
\label{disintegration}

The geometric definition of the Brownian burglar -- via its driver -- provides a clean, pathwise construction. We now finally motivate this definition by showing that it coincides exactly with the physical construction of the burglar obtained by disintegrating a Brownian motion embedded in a loop-soup with respect to its total occupation field (when $\nu \ge 0$), and also give another similar disintegration for the burglars of negative intensity.

\subsection{Burglars of nonnegative parameter}

Let $x, y \in \R$, and let $b \geq 0$. Let $\nu \ge 0$ be the intensity parameter. Let $\mathcal{L}$ be a Brownian loop-soup of intensity $\nu$ on $\R \setminus \{y\}$, and let $\Lambda$ be its occupation time profile. Let $B$ be an independent standard Brownian motion started from $x$ and stopped at the random time $\tau_{y }^b = \inf\{t \ge 0 : L_B(t, y ) > b\}$, where $L_B$ is the local time of $B$.
Let $\lambda_0$ be the total occupation profile of the field
$$\lambda_0(z) := \Lambda(z) + L_B(\tau_{y}^b, z), \quad z \in \R.$$
Let $I$ be the connected component of $\{\lambda_0 > 0\}$ containing $x$. By standard properties of squared Bessel processes, $\lambda_0$ is almost surely an admissible capacity on the interval $I$.

\begin{theorem}[Brownian motion/Loop-Soup Disintegration] \label{thm:disintegration}
Conditionally on the total occupation profile $\lambda_0$, the Brownian motion $B$ has the law of the burglar of intensity $\nu$, starting from $x$, targeting $y$, with initial capacity $\lambda_0$, parametrized in its quadratic variation time.
\end{theorem}

\begin{proof}
We prove this via the characterization of the burglar. We can define $\Pb^{\lambda}_{x \to y}$ the law of the Brownian motion $B$ conditioned on the occupation time field $\lambda$ for almost every $\lambda$ in the support of the random total occupation time field $\lambda_0$.

Almost surely, the path $B$ is admissible at $y$ for the capacity $\lambda_0$. The symmetry and the local-time Markov property are direct by symmetry and Markov property of the usual Brownian motion.

Stretching covariance is proven in Proposition 5.4 of \cite{AHS} in the boundary target case via their local-time flow approach. 

The side-locality of $B$ conditionally on $\lambda_0$ is a consequence of usual Brownian rewiring techniques : suppose e.g.\ $x<y$, then we can couple $B$ with a Brownian motion $B'$ started from $x$ and stopped at the time $\tau'$ at which it reaches some $y'>y$, and with a loop-soup $\mathcal L'$ of intensity $\nu$ in $\R\setminus \{y'\}$. Then, conditionally on $\lambda'_0(y) := \Lambda'(y) + L_{B'}(\tau', y)= b$, the local time profiles $\lambda_0$ and $\lambda'_0$ have same law below $y$ : they are both the sum of the local time profiles of a Brownian loop-soup of intensity $\nu$ in $(-\infty,y)$, an independent Brownian motion started at $x$ and when reaching $y$, and an independent Poisson point process of intensity $b$ times the measure on Brownian excursions below level $y$. Therefore, the couples $\left((B_t)_{0\leq t \leq \tau_y}, (\lambda_0(z))_{z\in(-\infty,y)}\right)$ and $\left((B'_t)_{0\leq t \leq \tau'_y}, (\lambda'_0(z))_{z\in(-\infty,y)}\right)$ have same law, which yields immediately the side-locality.

Using this side-locality property together with the stretching covariance for a boundary target gives stretching covariance for any target.

Lastly, at the final time, the Brownian path has consumed all the capacity available at the target, so the family is indeed maximal in time.

Therefore, by Theorem \ref{thm:classification} the family $(\Pb^{\lambda}_{x \to y})_{\lambda,x,y}$ (completed to any admissible capacity $\lambda$ in a well-defined way via stretching, thanks to stretching covariance) is the family of the laws of the Brownian burglars of some intensity $\alpha\in\R$.

Finally, the parameter $\alpha$ is determined in Proposition 5.13 of \cite{AHS}, where it is shown to be equal to the loop-soup intensity $\nu$, the intensity of the loop-soup, in boundary target case. By uniqueness of this parameter in the characterization of burglars, this case is sufficient to identify the laws $(\Pb^{\lambda}_{x \to y})_{\lambda,x,y}$ as the law of the Brownian burglars of intensity $\nu$.

For the sake of the completeness of our proof, let us give another short way to identify the parameter $\alpha$. Let e.g.\ $x=0$ and $y=1$, and $b=0$. We already have proved that conditionally on some $\lambda_0$, $B$ has the law of a burglar of intensity $\alpha$, starting from $0$, targeting $1$, with initial capacity $\lambda_0$. By the phases behaviour Proposition \ref{phases}, we have $\alpha > 0$ (unless $\nu=0$, in which case it is easy to get $\alpha=0$).
Let $g = \eta_{\lambda_0,\,0}$ be the stretching morphism uniformizing the capacity $\lambda_0$ to the uniform capacity 1. The driver of $B$ for initial stretching $g$ is the burglar driver of parameter $\alpha$ started from 0 and targeting $+\infty$, solution to the SDE
$$\dd D_t=\dd W_t+(\alpha-1)\dd t$$
for some Brownian motion $W$. 
By Lemma \ref{lemmatech}, as the time goes to its maximum, the local time at $x=0$ of $X:=g(B)$, $L_X(T(X),0) =  1 - \exp(-\ell_\infty(0))$, where $\ell_t(x)$ is the local time at 0 of the process $R_t(x) = D_t-g_t(g^{-1}(x))$. In particular $\ell_\infty(0)$ has the law of the local time at 0 of $ B_t+\alpha t$ as $t\to \infty$, which is famously an exponential random variable of parameter $\alpha$. Therefore, $L_X(T(X),0)= 1 - \exp(-\ell_\infty(0)) \sim \text{Beta}(1,\alpha)$. But we also know that $$L_X(T(X),0) = \frac{L_B(\tau^0_{1},0)}{\lambda_0(0)} = \frac{L_B(\tau^0_{1},0)}{L_B(\tau^0_{1},0)+ \Lambda(0)},$$
with $L_B(\tau^0_{1},0)\sim \Gamma(1,1/2)$ (as a squared Bessel process of dimension 2 evaluated at the point $x=1$) and $\Lambda(0)\sim \Gamma(\nu,1/2)$ (as a squared Bessel process of dimension $2\nu$ evaluated at the point $x=1$) independently. Therefore, $L_X(T(X),0)\sim  \text{Beta}(1,\nu)$, and finally $\alpha = \nu$.
\end{proof}

This Brownian motion / Loop-soup disintegration point of view also allows us to understand the phases of the burglar for $\nu \ge 0$ : it is usual that for $\nu\ge 1$, the loops of the loop-soup $\mathcal L$ in $\R\setminus\{y \}$ almost surely go through every point of $\R\setminus\{y \}$, such that the occupation time field $\Lambda$ is strictly positive on $\R\setminus\{y \}$, and thus $L_B(\tau^b_{y},x)<\lambda_0(x)$ for all $x\neq y$. For $0<\nu<1$, it is known that the occupation time field $\Lambda$ almost surely has infinitely many zeros near the point $y$, but any deterministic point $x\neq y$ is almost surely not a zero of $\Lambda$, such that for some fixed $x\in I$, $L_B(\tau^b_{y},x)<\lambda_0(x)$ with probability 1, but there also almost surely exists (infinitely many) exceptional points $x\in I$ such that $L_B(\tau^b_{y},x)=\lambda_0(x)$. Finally, for the intensity $\nu=0$, there is no loop in the loop-soup, and the Brownian motion conditioned by its local-time profile clearly exhausts all the available local time.

\ms

Moreover, by setting $b=0$, we get the disintegration in the boundary target case $y\in \partial I$. It is also possible to make sense of the boundary starting point $x\in\partial I$ via the same disintegration, by formally conditioning $\lambda_0$ to vanish at $x$, namely by taking $\mathcal L$ loop-soup in $\R\setminus\{x,y\}$ and $B$ a Brownian motion started at $x$ and conditioned not to come back to $x$ (i.e.\ $B = x + W$, where $W$ is Bessel process of dimension 3 started at 0). In particular, let $x<y$, $\mathcal L$ be a loop-soup of intensity $\nu$ in $(x,y)$ and $B$ be a Brownian excursion from $x$ to $y$ (staying in $(x,y)$), and let $\lambda_0$ the total occupation time of $\mathcal L$ and $B$. Then the burglar of intensity $\nu$, initial capacity $\lambda_0$, started from $x$ and targeting $y$ (in its quadratic variation time-parametrization) has exactly the law of the excursion $B$ conditionally on $\lambda_0$.

\subsection{Burglars of negative parameters}

Let $x, y \in \R$, and let $b \geq 0$. Let $\nu \leq 0$ be the intensity parameter.
Let $W$ be a $\nu$-doubly perturbed Brownian motion (DPBM) started (and centered) at $y$, as studied in \cite{MR1465164,MR1717529
}. $W$ is defined as the unique pathwise solution to  
$$ W_t = y+B_t +\nu\bigl((\sup_{s\leq t} W_s-y) +(\inf_{s\leq t} W_s-y)\bigr)$$
for $B$ a usual Brownian motion.
Let us stop $W$ at time $\tau_x^b :=\inf\{t\geq 0, L_W(t,x)>b\}$ (which is almost surely finite), and consider the time-reversed process $(\tilde W_t )_{0\leq t \leq \tau_x^b}=( W_{\tau_x^b -t})_{0\leq t \leq \tau_x^b}$. Let $\lambda_0 := L_W(\tau_x^b ,\cdot)$.

\begin{proposition}[DPBM Disintegration]
  \label{dpbm} 
  Conditionally on $\lambda_0$, the law of the time-reversed DPBM $\tilde W$ is that of the burglar of intensity $\nu\leq 0$, starting from $x$, targeting $y$, with initial capacity $\lambda_0$, parametrized in its quadratic variation time.
  \end{proposition}

  \begin{proof}
    We will follow a very similar path as for the positive case, and thus give less details. Let us verify the hypotheses of Theorem~\ref{thm:classification}.

    The Ray-Knight descriptions of the local times of PRBMs (see \cite{MR2454984}) show that, as a squared Bessel process, $\lambda_0$ is an admissible capacity. The process $\tilde W$ is admissible for $\lambda_0$ at the point $y$.
    
    Symmetry is immediate by reflecting the DPBM and its local times. Similarly as in Theorem \ref{thm:disintegration}, side-locality follows from usual Brownian rewirings and independence of excursions that can be deduced from the works \cite{MR1193919,MR1655296,MR1730618}.
    The local-time Markov property also follows from the same kind of rewirings, and the fact that the couple $(W_t, L_W(t,\cdot))$ is Markovian. 
    
    Stretching covariance of the DPBM conditionally on its local time follows from the stretching covariance of the PRBM shown in \cite{AHS} (for their burglar $Z^{(1)}$), and the fact that the excursions above and below level $y$ are independent.

    At final time $T$, $\lambda_T=L_W(0,\cdot)=0$, so no admissible continuation is possible and the family is indeed maximal in time. Theorem~\ref{thm:classification} therefore identifies the family of laws of the DPBM conditionally on $\lambda_0$ with the burglar family of some parameter $\alpha\in\R$.
    
    It remains to identify $\alpha$, which can also be done in a similar way as before. Let us e.g.\ consider the limitting case of $x=y=1$ being a boundary point (also described in more details later in the Subsection). In this case, $W$ can be seen for instance as an excursion of a $\nu$-PRBM from $y>0$ to $-\infty$ (staying in $(-\infty,y)$), followed by an independent Brownian excursion from $-\infty$ to $y$ (also staying in $(-\infty,y)$). Thus, at time $\tau_{-\infty}$ when $W$ reaches $-\infty$, the fraction of the total local time already exhausted at the point 0 by $W$ is
    $$\frac{L_W(\tau_{-\infty},0)}{\lambda_0(0)} = \frac{L_W(\tau_{-\infty},0)}{L_W(\tau_{-\infty},0)+ (\lambda_0(0)-L_W(\tau_{-\infty},0))}\sim \frac{\Gamma(1-\nu,1/2)}{\Gamma(1-\nu,1/2)+ \Gamma'(1,1/2)},$$
    which has law $\text{Beta}(1-\nu,1)$, and for the time-reversed process $\tilde W$, the corresponding consumed fraction is therefore $\text{Beta}(1,1-\nu)$ ; while looking at the driver and using Lemma \ref{lemmatech} as previously delivers us that the local time at 0 of a burglar of negative parameter $\alpha\leq 0$ on the constant capacity 1, started at $+\infty$ and targeting $+\infty$ is a $\text{Beta}(1,1-\alpha)$, finally allowing us to identify $\alpha=\nu$.
    \end{proof}

This disintegration also allows us to understand the last case of the phases of the burglars, indeed for all $\nu\leq 0$, the burglars always fill all the available local-time since the corresponding conditioned DPBM amount for all the local time of $\lambda_0$.

\ms

Moreover, similarly as for nonnegative intensity $\nu\geq 0$, by setting $b=0$, we get the disintegration in the boundary starting point case $x\in \partial I$. It is this time also possible to make sense of a boundary target $y\in\partial I$ via the same disintegration, by formally conditioning $\lambda_0$ to vanish at $y$, namely by taking $W$ a $\nu$-perturbed Brownian motion started at $y$ and conditioned to reach $x$ without coming back to $y$ (this can be done in a similar way as for Brownian motion, see \cite{MR1655297}). 

In particular, let $x<y$, $W$ be a $\nu$-perturbed Brownian excursion from $y$ to $x$ (staying in $(x,y)$), and let $\lambda_0$ be the total occupation time of $W$. Then the burglar of intensity $\nu$, initial capacity $\lambda_0$, started from $x$ and targeting $y$ (in its quadratic variation time-parametrization) has exactly the law of the time-reversed excursion $\tilde W$ conditionally on $\lambda_0$. This has particular interest, since by \cite{Titus_Lupu_2018}, a $\nu$-PRBM is exactly the concatenation of all the loops of a Brownian loop-soup of intensity $(1-\nu)\ge 1$, ordered by their minimum.

Thus, another way to obtain this burglar is the following : let $\nu\leq 0$, $x,y \in\R$, $\mathcal L$ be a loop-soup of intensity $(1-\nu)\geq 1$ in $(x,y)$, and $\lambda_0$ its occupation time. Then the burglar of intensity $\nu$, initial occupation profile $\lambda_0$, started from $x$ and targeting $y$ has exactly the law of the discovery of all the loops in $\mathcal L$, by increasing order of their minima in $(x,y)$, conditionally on the total occupation time $\lambda_0$.

Finally, if we also add to $\mathcal L$ an independent Brownian excursion from level $y$ to $x$ (staying in $(x,y)$), then the burglar of intensity $\nu$, initial occupation profile $\lambda_0$ (the total occupation time of both the loop-soup and the excursion), started from $y$ and targeting $y$ has exactly the law of the concatenation of first the Brownian excursion from $y$ to $x$, and then the discovery of all the loops in $\mathcal L$, conditionally on $\lambda_0$. This echoes the duality between burglars of intensity $\nu$ and $(1-\nu)$ highlighted in Subsection \ref{sidelocality}.

For $\nu=-\delta/2\in(-1,0)$, these exact burglars should also correspond to the local-time-filling spindle exploration processes recently studied in \cite{aïdékon2025stochasticflowsmarkedstable}. These explorations fill local times which are squared Bessel processes of negative dimension $-\delta$, see their Section 3 for more details.

\ms

This is all also related to the decomposition given in \cite{pitman2018squaredbesselprocessespositive}, where they decompose a Brownian path into excursions above and below an evolving level. The concatenation of the excursions below give a PRBM (thus corresponding to our burglar), which can then be understood in some sense as a full usual Brownian motion on top of a "negative intensity loop-soup" corresponding to the excursions above that have been cut out (and indeed yield as local time a squared Bessel process of negative dimension). This is exactly the right point of view to understand why the PRBM/DPBM construction is the natural negative-parameter analogous of Subsection \ref{thm:disintegration}.

\begin{remark}
  The DPBM is actually well-defined for all $\nu<1$. However, for $\nu\in(0,1)$, the corresponding time-reversed is not an admissible process for its total local time $\lambda_0$ (since for these intensities, there are exceptional points through which the DPBM goes but spends 0 local time). That is the only step at which the characterization of Theorem~\ref{thm:classification} in Proposition \ref{dpbm} fails. The burglars of parameter $\nu\in(0,1)$ are still closely related to the corresponding DPBMs. Indeed, as mentioned in the introduction, we believe that it should be straightforward to translate the construction of the space-filling SLEs for $4<\kappa<8$ into a burglar analogous, namely ``local-time filling'' burglars for $0<\nu<1$, which would then exactly correspond to the DPBM disintegration for these intensities. Another way to say this is that the intermediate-regime ($\nu\in(0,1)$) burglars (which can be understood via loop-soup disintegration, Theorem \ref{thm:disintegration}) also correspond to (time-reversed) DPBMs, restricted to the times when they are in the connected component of $x$ in $\{z\in\R, L_W(t,z)>0\}$, and conditioned on $\{\tau_x^b<+\infty\}$.
\end{remark}

\subsection{Recovering the whole loop-soup}
\label{loopsoup}

The burglars that we just defined, and in particular the loop-soup disintegration point of view for $\nu >0$, give us a way to recover the law of the loop-soup $\mathcal L_\nu$ of intensity $\nu>0$ on a bounded interval $I$ of $\R$ from its occupation time field $\Lambda$.
The algorithm is the following :
\begin{itemize}
\item Choose any dense sequence $(x_n)$ of points in $I$, and set $\Lambda^{(0)} := \Lambda$.
\item For $n \geq 0$, run $(X^n_t)_{0\leq t \leq T_n}$ a burglar of parameter $\nu$ started from $x_n$ and targeting $x_n$, with initial capacity $\Lambda^{(n)}$.
\item Set $\Lambda^{(n+1)} := \Lambda^{(n)} - L_{X^n}(T_n,\cdot)$. Then go back to the previous step.
\end{itemize}

As $n\to \infty$, this procedure creates burglars that will exhaust all the available local time $\Lambda$ in the interval $I$. It only remains to cut their paths into separate loops to recover a loop-soup. This cutting must be done at the point $x_n$ for the path $X^n$, via the appropriate cutting procedure, analogous to the stick-breaking algorithm at critical intensity. The path $X^n$ spends local time $\Lambda^{(n)}(x_n)$ at the point $x_n$, we index the excursions of $X^n$ from $x_n$ by the local time at $x_n$ at which they appear. The path $X^n$, seen as a concatenation of excursions is then cut into an infinite number of loops. The law of the proportion of the local time of the first loop we discover is $\text{Beta}(1,\nu)$ -- see \cite{LeJan2011,Titus_Lupu_2018}. The first loop has local time $\Lambda^{(n)}(x_n) Y_1$, the second one $\Lambda^{(n)}(x_n) (1-Y_1)Y_2$, and so on, with $Y_i\sim \text{Beta}(1,\nu)$ i.i.d.\ -- see also \cite{werner2015spatialmarkovpropertysoups,werner2025switchingidentitycablegraphloop} for an explanation of the stick-breaking algorithm for discrete loops, and at critical intensity for Brownian loops.
The collection $\mathcal L$ of loops obtained in this way then has the law of the loop-soup $\mathcal L_\nu$ conditionally on its occupation time field $\Lambda$. 

\ms

Finally, one last thing worth mentioning is that via this disintegration, properties proven for the burglars can actually be passed to the whole loop-soup, in particular, on the line, the Brownian loop-soups of any intensity $\nu>0$ are stretching covariant. Indeed, for any $\nu>0$ and any admissible capacity $\lambda$ we can make sense of the conditional law of the Brownian loop-soup $\mathcal L_\nu(\lambda)$ of intensity $\nu$ conditioned to have occupation time profile $\lambda$ ; and then for $g$ a stretching morphism from $\lambda$ to $\rho$, the stretched collection $g(\mathcal L_\nu(\lambda)) := \{g(L), L\in \mathcal L_\nu(\lambda)\}$ has the same law as $\mathcal L_\nu(\rho)$.

\subsection{Flow of local times}

Let us now conclude with a quick discussion about the local time flows of the burglars explaining how the results of \cite{AHS} naturally extend to negative intensities and general targets.

The disintegration point of view for both $\nu\geq 0$ and $\nu \leq 0$ leads to a general unified description of the local time flow of the burglar $X$ of parameter $\nu\in\R$ on an admissible capacity $\lambda$, started from $x\in\overline I$ and targeting $y\in\overline I$. For some fixed $a\in I$, let $\eta:z\in I\mapsto \int_a^z\frac{\dd r}{\lambda(r)}$. In our vocabulary, it is a stretching morphism uniformizing capacity $\lambda$ to the constant capacity 1.

Then, for all $\nu\in\R$, the local time flow of the stretched process $\eta(X)$ (defined up to time parametrization) is embedded inside a $\Jac(2\,|_{\eta(x)}\,0,\, 2-2\nu\,|_{\eta(y)}\, 2\nu)$ flow -- again, see \cite{AHS} for definitions and properties. More precisely, the local time flow of $\eta(X)$ has the law of a Jacobi flow $$Y\sim\Jac(2\,|_{\eta(x)}\,0,\, 2-2\nu\,|_{\eta(y)}\, 2\nu)$$ under its ``primitive Eve'' 
$$ \Xi_r := \inf \{ v \in [0, 1] : Y_{r, \eta(y)}(v) = 1\},$$
i.e.\ for $L$ the local time process of $\eta(X)$, $\tau_b^r = \inf\{t, L(t,r)>b\}$ and $\mathcal S_{r,z}(b) = L(\tau_b^r,z)$, then 
$$ \bigl(\mathcal S_{r,z}(b) \bigr)_{r,z\in \R,\,0\leq b\leq L(\infty,r)} \sim  \bigl(Y_{r,z}(b) \bigr)_{r,z\in \R,\,0\leq b\leq \Xi_r}. $$
Notice that for $\nu\leq 0$, since the process fills all the available local time, the primitive Eve is trivial $\Xi_r=1$ for all $r\in\R$ and the local time flow of our process is the whole Jacobi flow.

\bs

\subsection*{Acknowledgements}
I thank Wendelin Werner for help and advice during the preparation of this paper. I also thank \'Elie Aïdékon for enlightening discussions, and useful feedback on a first version of this paper. This research has been funded by a grant from the Royal Society.

\end{document}